\documentclass[final]{siamart190516}

\usepackage{amsfonts}
\usepackage{graphicx}
\usepackage{epstopdf}
\usepackage{algorithmic}
\usepackage{xcolor}
\ifpdf
  \DeclareGraphicsExtensions{.eps,.pdf,.png,.jpg}
\else
  \DeclareGraphicsExtensions{.eps}
\fi

\newsiamremark{remark}{Remark}
\newsiamremark{hypothesis}{Hypothesis}
\crefname{hypothesis}{Hypothesis}{Hypotheses}
\newsiamthm{claim}{Claim}
\newsiamthm{example}{Example}
\newsiamthm{assumption}{Assumption}
\crefname{assumption}{Assumption}{Assumptions}

\headers{RS-RNM for the unconstrained non-convex optimization
}{T. Fuji, P. L. Poirion, and A. Takeda}

\title{Theoretical analysis of the randomized subspace regularized Newton method for non-convex optimization
\thanks{
\funding{This work was supported by JSPS KAKENHI (17H01699 and 19H04069) and JSTERATO (JPMJER1903).}}}

\author{Terunari Fuji\footnote{The first two authors contributed equally to this work.}\thanks{ Graduate School of Information Science and Technology, The University of Tokyo, Tokyo 113-8656, Japan.
  (\email{teru0818258@g.ecc.u-tokyo.ac.jp}).}
\and Pierre-Louis Poirion\thanks{Center for Advanced Intelligence Project, RIKEN, Tokyo 103-0027, Japan.
  (\email{pierre-louis.poirion@riken.jp}).}
\and Akiko Takeda\thanks{Graduate School of Information Science and Technology, The University of Tokyo, Tokyo 113-8656, Japan and
Center for Advanced Intelligence Project, RIKEN, Tokyo 103-0027, Japan.
  (\email{takeda@mist.i.u-tokyo.ac.jp}).}
}

\usepackage{amsopn}

\usepackage{amsmath,amssymb,bm}
\usepackage{array,arydshln}

\usepackage{comment}

\usepackage{breqn}

\usepackage{subfigure}

\newcommand{\eps}{\varepsilon}
\newcommand{\R}{\mathbb R}

\newcommand{\norm}[1]{\left\|#1\right\|}

\newcommand{\T}{{\mathsf T}}

\newcommand{\C}{{\mathcal C}}

\DeclareMathOperator{\Range}{Range}
\DeclareMathOperator*{\argmin}{arg~min}

\usepackage{ifthen}
\newcommand{\COMM}[2]{{
\ifthenelse{\equal{#1}{TF}}{\color{red}}{
\ifthenelse{\equal{#1}{PL}}{\color{blue}}{
\ifthenelse{\equal{#1}{AT}}{\color{magenta}}}}[#1: #2]}}

\usepackage{tcolorbox}
\tcbuselibrary{skins, breakable}

\usepackage[nameinlink]{cleveref}

\colorlet{linkcolor}{green!50!black}
\colorlet{exlinkcolor}{black}

\hypersetup{
  colorlinks = true,
  allcolors = linkcolor,
  urlcolor = exlinkcolor,
  frenchlinks=false,
  pdfborder={0 0 0},
  naturalnames=false,
  hypertexnames=false,
  breaklinks
}
\usepackage{cleveref}
\usepackage{enumitem}
\newcommand{\rank}{\mathrm{rank}}
\newcommand{\Ker}{\mathrm{Ker}}
\newcommand{\Img}{\mathrm{Im}}
\ifpdf
\hypersetup{
  pdftitle={Theoretical analysis of the randomized subspace regularized Newton method for non-convex optimization},
  pdfauthor={T. Fuji, P. L. Poirion, and A. Takeda}
}
\fi

\begin{document}

\maketitle

\begin{abstract}
While there already exist randomized subspace Newton methods that restrict the search direction to a random subspace for a convex function, 
we propose a randomized subspace regularized Newton method for a non-convex function {and more generally we investigate thoroughly, for the first time, the local convergence rate of the randomized subspace Newton method}. 
In our proposed algorithm, we use a modified Hessian of the function restricted to some random subspace so that, 
with high probability, the function value decreases at each iteration, even when the objective function is non-convex.
We show that our method has global convergence under appropriate assumptions and its convergence rate is the same as that of the full regularized Newton method. 
Furthermore, we obtain a local linear convergence rate under some additional assumptions, and prove that this rate is the best we can hope, in general, when using a random subspace. We furthermore prove that if the Hessian, at the local optimum, is rank deficient then super-linear convergence holds.
\end{abstract}



\section{Introduction} \label{sec:introduction}
While first-order optimization methods such as stochastic gradient descent methods are well studied for large-scale machine learning optimization,
second-order optimization methods have not received much attention due to the high cost of computing second-order information until recently.
However, in order to overcome relatively slow convergence of first-order methods,
there has been recent interest in second-order methods that aim to achieve faster convergence speed by utilizing subsampled Hessian information and stochastic Hessian estimate
(see e.g., \cite{Bottou18,Xu20,Yao21} and references therein). 

In this paper, we develop a Newton-type iterative method with random projections for the following unconstrained optimization problem:
\begin{equation}\label{eq:main_problem}
  \min_{x\in\R^n} f(x),
\end{equation}
where $f:\R^n \to \R$ is a possibly non-convex twice differentiable function. In our method, at each iteration, we restrict the function $f$ to a random subspace and compute the next iterate by choosing a descent direction on this random subspace.

There are some existing studies on developing second-order methods with random subspace techniques  for convex optimization problems \eqref{eq:main_problem}. 
Let us now review randomized subspace Newton (RSN) existing work \cite{gower2019rsn}, while gradient-based
randomized subspace algorithms are reviewed in Section~\ref{sec:relatedworks}.
RSN  
computes the descent direction $d_k$ and the next iterate  as
\begin{align*}
    d_k^{\rm RSN} &= -P_k^\T(P_k\nabla^2 f(x_k)P_k^\T)^{-1}P_k\nabla f(x_k),\\
    x_{k+1} &= x_k + \dfrac{1}{\hat L}d_k^{\rm RSN},
\end{align*}
where $P_k\in \R^{s\times n}$ is a random matrix with $s < n$ and $\hat L$ is some fixed constant.
RSN is expected to be highly computationally efficient with respect to the original Newton method, since it does not require computation of the full Hessian inverse.
RSN is also shown to achieve a global linear convergence for strongly convex $f$. We first note that the second-order Taylor approximation around $x_k$ restricted in the affine subspace $\{x_k\} + \Range(P_k^\T)$ is expressed as
\begin{equation*}
  f(x_k +P_k^\T u) \simeq f(x_k) + \nabla f(x_k)^\T P_k^\T u + \frac{1}{2} u^\T P_k\nabla^2 f(x_k)P_k^\T u,
\end{equation*}
and the direction $d_k^{\rm RSN}$ is obtained as $d_k^{\rm RSN} = P_k^\T u_k^*$ where $u_k^*$ is the minimizer of the above subspace Taylor approximation, i.e., 
\begin{equation*}
  u_k^* = \argmin_{u\in\R^s} \left(f(x_k) + \nabla f(x_k)^\T P_k^\T u + \frac{1}{2} u^\T P_k\nabla^2 f(x_k)P_k^\T u\right).
\end{equation*}
Hence, we can see that the next iterate of RSN is computed by using the Newton direction for the function
\begin{equation}\label{eq:randfun}
f_{x_k}\ :\ u \mapsto f(x_k+P_k^\top u).
\end{equation}
Other second-order subspace descent methods, such as cubically-regularized subspace Newton methods, \cite{cubicNewton}, have been studied in the literature. More precisely, the method in \cite{cubicNewton} can be seen as a stochastic extension of the cubically-regularized Newton method \cite{nesterov2006} and also as a second-order enhancement of stochastic subspace descent \cite{kozak2019stochastic}. In \cite{kovalev2020}, a random subspace version of the BFGS method is proposed. The authors prove local linear convergence, if the function is assumed to be self-concordant.
Apart in recent Shao's Ph.D thesis \cite{Shao} and the associated papers \cite{GNC,CJZ} which have been done parallelly to this paper, to the best of our knowledge, 
 existing second-order subspace methods have iteration complexity analysis only for convex optimization problems.
 
The thesis \cite{Shao} and the paper \cite{CJZ} propose a random subspace adaptive regularized cubic method for unconstrained non-convex optimization and show
 a global convergence property with sub-linear rate to a stationary point\footnote{The author also proves that if the Hessian matrix has low rank and scaled Gaussian sketching matrices are used, then the Hessian at the stationary point is approximately positive semidefinite with high probability.}.
 In this paper we propose a new subspace method based on the regularized Newton method and discuss the local convergence rate together with
 global iteration complexity.\footnote{Just as the ordinary cubic method is superior to the Newton method in terms of iteration complexity, similar observation seems to hold between the subspace cubic method \cite{Shao} and ours.} 
 {Notice indeed that, to the best of our knowledge, the local convergence of such methods never seems to have been thoroughly studied\footnote{Some papers, as we will see later, investigate when local linear convergence holds.}; one would expect super-linear convergence for second order methods and no papers discuss whether super-linear convergence holds or not for second order methods. Indeed any iterative algorithm can be easily adapted to a random subspace method as it suffices to apply it to the function  restricted to the subspace: $u \mapsto f(x_k+P^\top_ku)$. We therefore believe that it is important to investigate thoroughly whether the properties of such full-space algorithms are preserved or not when adapted to the random subspace setting.}


If the objective function $f$ is not convex, the Hessian is not always positive semidefinite and $d_k^{\rm RSN}$ is not guaranteed to be a descent direction 
so that we need to use a modified Hessian.
Based on the regularized Newton method (RNM) for the unconstrained non-convex optimization \cite{ueda2010convergence, ueda2014regularized},
we propose the randomized subspace regularized Newton method (RS-RNM):
\begin{align*}
  d_{k} &= -P_k^\T(P_k\nabla^2 f(x_k)P_k^\T + \eta_k I_s)^{-1}P_k\nabla f(x_k),\\
  x_{k+1} &= x_k + t_k d_k,
\end{align*}
where $\eta_k$ is defined to ensure that search direction $d_k$ is a descent direction and the step size $t_k$ is chosen so that it satisfies Armijo's rule.
As with RSN, this algorithm is expected to be computationally efficient since we use projections onto lower-dimensional spaces.
In this paper, we first show that
RS-RNM has global convergence under appropriate assumptions, more precisely, we have $\norm{\nabla f(x_k)} \le \eps$ after at most $O(\eps^{-2})$ iterations with some probability, 
which is the same as the global convergence rate shown in \cite{ueda2010convergence} for the full regularized Newton method. We then prove that under additional assumptions, we can obtain a linear convergence rate locally. In particular, one contribution of the paper is to propose, to the best of our knowledge, the weakest conditions until now for local linear convergence. To do so we will extensively use the fact that the subspace is chosen at random. From these conditions, we can derive a
random-projection version of the Polyak-Lojasiewicz (PL) inequality \eqref{eq:pl},
\begin{equation}\label{eq:pl}
	\forall x\in \mathbb{R}^n, \quad \|\nabla f(x)\|^2 \ge c_0 (f(x)-f(x^*)),
\end{equation}
 which will be satisfied when the function is restricted to a random subspace. One other contribution of this paper is to prove that, in general, linear convergence is the best rate we can hope for this method.{ Furthermore, we also prove that if the Hessian at the local optima is rank deficient, then one can achieve super-linear convergence using a subspace dimension $s$ large enough.}

 {Our randomized subspace method for nonconvex optimization problems is based on the regularized Newton method in \cite{ueda2010convergence, ueda2014regularized}. While various other regularized Newton methods have been proposed in recent years, most of them are for convex problems or non-smooth optimization problems.
   For example, \cite{mordukhovich2023} presents a globally convergent proximal Newton-type method for non-smooth convex optimization and \cite{chao2024} develops coderivative-based Newton methods combined with Wolfe line-search for non-smooth problems. Recently \cite{yamakawa} proposes a generalized regularization method that includes quadratic, cubic, and elastic net regularizations. Also \cite{Doikov24} proposes, in the convex case, a variant of the Newton method with quadratic regularization and proves better global rate. Recent papers, \cite{gratton,zhou2025,zhu2024}, propose regularization methods for the non-convex case. However, although these methods obtained better iterations complexity, the subroutines involved to compute are quite complex and not as simple as in \cite{ueda2010convergence, ueda2014regularized}.
   By applying similar random subspace techniques to these methods, we may be able to develop random subspace variants with state-of-the-art theoretical guarantees, but that is a topic for future work. }

 The rest of this paper is organized as follows. After reviewing gradient-based
randomized subspace algorithms and introducing properties of random projections in \Cref{sec:preliminaries},
we introduce our random subspace Newton method for non-convex functions in~\Cref{sec:RSRNM_algorithm}. In~\Cref{sec:RSRNM_global}, 
we prove global convergence properties for our method. {In~\Cref{sec:local1}, we investigate local linear convergence as well as local super-linear convergence. Finally, in~\Cref{sec:num}, we show some numerical examples to illustrate the theoretical properties derived in the paper}. In~\Cref{sec:conclusions} we conclude the paper.

\section{Preliminaries}\label{sec:preliminaries}

\paragraph{Notation:}
In this paper we call a matrix $P \in \R^{s\times n}$ a random projection matrix or a random matrix when its entries $P_{ij}$ are independently sampled from the normal distribution $\mathrm N (0, 1/s)$. Let $I_n$ be the identity matrix of size $n$. We denote by $g_k$ the gradient of the $k$-th iterate of the obtained sequence and by $H_k$ it's Hessian.

\subsection{Related optimization algorithms using random subspace} \label{sec:relatedworks}
As introduced in Section~\ref{sec:introduction}, random subspace techniques are used for second-order optimization methods
with the aim of reducing the size of Hessian matrix.
Here we refer to other types of subspace methods focusing on their convergence properties.

{Cartis et al. \cite{cartis2023global} proposed a general framework to 
	investigate a general random embedding framework for global optimization of a function $f$. 
	The framework projects the original problem onto a random subspace and  solves the reduced subproblem in each iteration:
	\[
	\min_{u} f(x_{k} + P_k^\top u) ~~ \mbox{subject to}~ x_{k} + P_k^\top u \in \mathcal{C}.
	\]
	These subproblems need to be solved to some required accuracy by using a deterministic global optimization algorithm. This study is further expanded in \cite{cartisOtessimov} and \cite{cartis2022bound}, when $f$ has low effective dimension.} 

There are also various subspace first-order methods based on coordinate descent methods (see e.g. \cite{wright2015coordinate}).
In \cite{chen2020randomized} a randomized coordinate descent algorithm is introduced assuming some subspace decomposition which is suited to the $A$-norm, where $A$ is a given preconditioner.
In \cite{AdaptiveRand},  minimizing $f(\tilde{A}x)+\frac{\lambda}{2}\|x\|^2$, where $f$ is a strongly convex smooth function
and $\tilde{A}$ is a high-dimensional matrix, is considered and a new randomized optimization method that can be seen as a generalization of coordinate descent to random subspaces is proposed. The paper \cite{grishchenko2021proximal} deals with a convex optimization problem $\min\limits_x f(x)+g(x)$, where $f$ is convex and differentiable and $g$ is assumed to be convex, non-smooth and sparse inducing such as $\|x\|_1$. To solve the problem, they propose a randomized proximal algorithm leveraging structure identification: the variable space is sampled according to the structure of $g$.
The approach in \cite{randomPursuit} is to optimize a smooth convex function by choosing, at each iteration a random direction on the sphere.
Recently, in some contexts such as iteration complexity analysis, the assumption of strong convexity has been replaced by
a weaker one, the PL inequality \eqref{eq:pl}. 
Indeed,
\cite{kozak2021stochastic} has introduced a new first-order random subspace
and proved that if the non-convex function is differentiable with a Lipschitz first derivative and satisfies the PL inequality \eqref{eq:pl} 
then linear convergence rate is obtained in expectation. 
Notice that in all these papers a local linear convergence rate is only obtained when assuming that the objective function is, at least locally, strongly convex or satisfies the PL inequality.

From above, without (locally) strong convexity nor the PL inequality, it seems difficult to construct first-order algorithms having (local) linear convergence rates.
Indeed, the probabilistic direct-search method \cite{roberts2022direct} in reduced random spaces is applicable to both convex and non-convex problems but it obtains sub-linear convergence.

In this paper, we will prove that our algorithm achieves local linear convergence rates without locally strong
convexity nor the PL inequality assumption on the full space. This is due to randomized Hessian information used in our algorithm. {More precisely,  our assumptions will allow us to prove that the function, restricted to a random subspace, satisfies a condition similar to the PL inequality. }

\subsection{Properties of random projection}
In this section, we recall basic properties of random projection matrices.
One of the most important features of a random projection defined by a random matrix 
is that it nearly
preserves the norm of any given vector with arbitrary high probability.
The following lemma is known as a variant of the Johnson-Lindenstrauss lemma  \cite{jllemma}.

\begin{lemma}[\protect{\cite[Lemma 5.3.2, Exercise 5.3.3]{vershynin2018high}}] 
  \label{lem:JLL_random_matrix}
  Let $P\in \R^{s\times n}$ be a random matrix
  whose entries $P_{ij}$ are independently drawn from $\mathrm N(0,1/s)$.
  Then for any $x\in\R^n$ and $\eps\in(0,1)$, we have
  $${\rm Prob\ }[(1-\eps)\norm{x}^2 \le \norm{Px}^2 \le (1+\eps)\norm{x}^2]
    \ge 1-2\exp(-\mathcal C_0 \eps^2 s), $$
  where $\mathcal C_0$ is an absolute constant.
\end{lemma}

The next lemma shows that when $P$ is a Gaussian matrix, we can obtain a bound on the norm of $PP^\top$.
\begin{lemma} \label{lem:concentration_of_PP}
	For a $s\times n$ random matrix $P$ whose entries are sampled from $\mathrm N (0, 1/s)$, there exists a constant $\bar{\mathcal{C}}>0$ such that 
	\begin{equation*}
		\norm{P P^\T} (=\norm{P^\T P} = \norm{P}^2) \le \bar{\mathcal{C}}\frac{n}{s},
	\end{equation*}
	with probability at least $1-2e^{-s}.$
\end{lemma}
\proof{Proof.}
	By \cite[Theorem 4.6.1]{vershynin2018high}, there exists a constant $\bar{C}$ such that
	\begin{equation*}
		\left\|\frac{s}{n}PP^\top -I_s\right\| \le 2\bar{C}\sqrt\frac{s}{n}
	\end{equation*}
	holds with probability at least $1-e^{-s}$. Therefore, we have 
	\begin{align*}
	  \norm{PP^\T} \le \norm{PP^\T - \frac{n}{s} I_s} + \norm{\frac{n}{s} I_s} \le 2\bar{C}\sqrt\frac{n}{s}+\frac{n}{s}\le  2\bar{C}\frac{n}{s}+\frac{n}{s} =
           (2\bar{C}+1)\frac{n}{s} .
	\end{align*}
	Setting $\bar{\mathcal{C}}=2\bar{C}+1$ ends the proof.
\endproof

All the results of this paper are stated in a probabilistic way. In the proofs we will constantly use the following fact:
\begin{equation}\label{eq:unionbound}
	\mbox{For any two events }E_1\mbox{ and }E_2:\ \rm{Prob}(E_1\cap E_2) \ge 1-\left((1-\rm{Prob}(E_1))+(1-\rm{Prob}(E_2))\right).
\end{equation}






\section{Randomized subspace regularized Newton method}\label{sec:RSRNM_algorithm}
In this section,
we describe a randomized subspace regularized Newton method (RS-RNM) for the following unconstrained minimization problem, 
\begin{equation}
  \min_{x\in \R^n} f(x),
\end{equation}
where $f$ is a twice continuously differentiable function from $\R^n$ to $\R$.
In what follows, we denote the gradient $\nabla f(x_k)$ and the Hessian $\nabla^2 f(x_k)$ as $g_k$ and $H_k$, respectively.

 
The paper \cite{ueda2010convergence} develops 
a regularized Newton methods (RNM) that constructs a sequence of iterates with the following update rule:
\begin{equation*}
  x_{k+1} = x_{k} - t_k(H_k + c_1'\Lambda'_k I_n + c_2'\norm{g_k}^{\gamma'}I_n)^{-1} g_k,
\end{equation*}
where $\Lambda'_k = \max(0,-\lambda_{\min}( H_k))$, $c_1', c_2', \gamma'$ are some positive parameter values and $t_k$ is the step-size chosen by Armijo's step size rule, 
and show that this algorithm achieves $\norm{g_k}\le\eps$ after at most $O(\eps^{-2})$ iterations and 
it has a super-linear rate of convergence in a neighborhood of a local optimal solution under appropriate conditions.

To increase the computational efficiency of this algorithm using random projections, 
based on the randomized subspace Newton method~\cite{gower2019rsn},
we propose the randomized subspace regularized Newton method (RS-RNM) with Armijo's rule,  which is described in \Cref{alg:RSN} and outlined below.  
Since RS-RNM is a subspace version of RNM, 
all discussions of global convergence guarantees made in \Cref{sec:RSRNM_global} are based on the one in \cite{ueda2010convergence}.

Let $\mathcal D$ denote the set of Gaussian matrices of size $s\times n$ whose entries are independently sampled from $\mathrm N (0, 1/s)$.
With a Gaussian random matrix $P_k$ from $\mathcal D$, 
 the regularized sketched Hessian is defined by:
          \begin{equation}\label{def:Mk}
            M_k := P_k H_k P_k^\T + \eta_k I_s \in \R^{s\times s},
          \end{equation}
          where $\eta_k:=c_1\Lambda_k  + c_2\norm{g_k}^\gamma$ and 
          $\Lambda_k := \max(0,-\lambda_{\min}(P_k H_k P_k^\T))$. We then compute the search direction:
          \begin{equation}\label{eq:direction}
            d_k := -P_k^\T M_k^{-1} P_k g_k.
          \end{equation}
          The costly part of Newton-based methods, the inverse computation of a (approximate) Hessian matrix, is done in the subspace of size $s$. We note that $d_k$ defined by \eqref{eq:direction} is a descent direction for $f$ at $x_k$, i.e.,
          $g_k^\top d_k < 0$ if $g_k \neq 0$, since it turns out that $M_k$ is positive definite from the definition of $\Lambda_k$, and therefore $x^\top P_k^\T M_k^{-1} P_k x >0$ holds for $\forall x$ due to $P_k x \neq 0$ with high probability.

    The backtracking line search with Armijo's rule finds the smallest integer $l_k \ge 0$ such that
    \begin{equation}\label{eq:Armijo_rule}
      f(x_k) - f(x_k + \beta^{l_k} d_k) \ge -\alpha \beta^{l_k} g_k^\T d_k.
    \end{equation}
Starting with $l_k=0$, $l_k$ is increased by $l_k \leftarrow l_k+1$ until the condition \eqref{eq:Armijo_rule} holds. The sufficient iteration number to find such a step-size is discussed in convergence analysis later.

 \begin{algorithm}[!t]
 	\caption{Randomized subspace regularized Newton method (RS-RNM)}
 	\label{alg:RSN}
 	\begin{algorithmic}[1]    
 		\renewcommand{\algorithmicrequire}{\textbf{input:}}
 		\REQUIRE $x_0\in\R^n$, $\gamma\ge0, c_1>1, c_2>0, \alpha, \beta\in (0,1)$
 		\STATE{$k \leftarrow 0$}
 		\REPEAT
 		\STATE{sample a random matrix: $P_k \sim \mathcal D$}
 		\STATE{compute the regularized sketched Hessian: $M_k = P_k H_k P_k^\T + c_1\Lambda_k I_s + c_2\norm{g_k}^\gamma I_s$, where  $\Lambda_k = \max(0,-\lambda_{\min}(P_k H_k P_k^\T))$}
 		\STATE{compute the search direction: $d_k = - P_k^\T M_k^{-1} P_k g_k$}
 		\STATE{apply the backtracking line search with Armijo's rule by finding the smallest integer $l_k \ge 0$ such that
 			\eqref{eq:Armijo_rule} holds. Set $t_k = \beta^{l_k}$, $x_{k+1}=x_k + t_k d_k$ and $k \leftarrow k+1$  }
 		\UNTIL{some stopping criteria is satisfied}
 		\RETURN the last iterate $x_k$
 	\end{algorithmic}
 \end{algorithm}

\section{Global convergence properties} \label{sec:RSRNM_global}
In \Cref{sec:RSRNM_global_convergence}, we discuss the global convergence of the RS-RNM under  \Cref{assumption:RS-RNM_global}. We further prove the global iteration complexity of the algorithm in \Cref{sec:RSRNM_global_complexity} by considering further assumptions.

\begin{assumption}\label{assumption:RS-RNM_global}
  The level set of $f$ at the initial point $x_0$  is {bounded}, i.e.,
  $\Omega := \{x \in \R^n: f(x) \leq f(x_0)\}$ is {bounded}.
\end{assumption}

By \cref{eq:Armijo_rule}, we have that for any $k\in \mathbb{N}$, $f(x_{k+1})\le f(x_k)$, implying all $x_{k} \in \Omega$.
Since $\Omega$ is a bounded set and $f$ is continuously differentiable, there exists $U_g>0$ such that 
    \begin{equation}\label{eq:upper_bound_g_k}
      \norm{g_k} \le U_g, \ \forall k \ge 0.
    \end{equation} 
Similarly, there exists $L>0$ such that for all $x\in \Omega$,  
\begin{equation}\label{eq:lip}
	\|\nabla^2f(x)\|\le L.
\end{equation}
In particular, for all $k>0$,
\begin{equation}\label{eq:lip1}
	\|H_k\|\le L.
\end{equation}
Notice that by \cref{eq:lip}, $\nabla f$ is $L$-Lipschitz continuous. We also define $f^* = \inf_{x\in \Omega} f(x)$.


\subsection{Global convergence}\label{sec:RSRNM_global_convergence}
We first show that the norm of $d_k$ can be bounded from above.
\begin{lemma}\label{lem:upper_bound_of_d_k_1}
  Suppose that $\norm{d_k}\neq 0$. Then, $d_k$ defined by \cref{eq:direction} satisfies
  \begin{equation*}
    \norm{d_k} \le \bar{\mathcal{C}}\frac{n}{s}\frac{\norm{g_k}^{1-\gamma}}{c_2},
  \end{equation*}
with probability at least $1-2e^{-s}.$
\end{lemma}
\proof{Proof.}
  By \Cref{lem:concentration_of_PP} we have $\norm{P_k^\T P_k} \le \bar{\mathcal{C}}\frac{n}{s}$, holds with probability at least $1-2e^{-s}$. Then, it follows from \cref{eq:direction} that
  \begin{align*}
    \norm{d_k} &= \norm{P_k^\T M_k^{-1} P_k g_k}\\
    &= \norm{P_k^\T (P_k H_k P_k^\T + 
      \eta_k I_s)^{-1} P_k g_k}\\ 
    &\le \norm{P_k^\T (P_k H_k P_k^\T + 
      \eta_k I_s)^{-1} P_k}\norm{g_k}\\
    &\le \norm{P_k^\T} \norm{ P_k}\norm{(P_k H_k P_k^\T + 
      \eta_k I_s)^{-1}} \norm{g_k}\\
    &= \frac{\norm{P_k^\T P_k}\norm{g_k}}{\lambda_{\min}(P_k H_k P_k^\T + c_1\Lambda_k I_s + c_2\norm{g_k}^\gamma I_s)} \quad (\mbox{as $\norm{P_k^\T} \norm{ P_k}=\norm{P_k^\T P_k}$})\\
               &\le \bar{\mathcal{C}}\frac{n}{s} \frac{\norm{g_k}^{1-\gamma}}{c_2}.
  \end{align*}
\endproof

We next show that, when $\norm{g_k}$ is at least $\varepsilon$ away from $0$, $\norm{d_k}$ is bounded above by some constant.
\begin{lemma}\label{lem:upper_bound_of_d_k_2}
  Suppose that \Cref{assumption:RS-RNM_global} holds. Suppose also that there exists $\eps>0$ such that $\norm{g_k}\ge \eps$.
  Then, with probability at least $1-2e^{-s}$, $d_k$ defined by \cref{eq:direction} satisfies
  \begin{equation} \label{eq:upper_bound_d_k_2}
    \norm{d_k} \le r(\eps),
  \end{equation}
  where 
  \begin{equation*}
    r(\eps) := \frac{\bar{\mathcal{C}}n}{c_2 s} \max\left(U_g^{1-\gamma}, \frac{1}{\eps^{\gamma-1}}\right).
  \end{equation*}
\end{lemma}
\proof{Proof.}
  If $\gamma \le 1$, it follows from \Cref{lem:upper_bound_of_d_k_1} and \cref{eq:upper_bound_g_k} that
  \begin{equation*}
    \norm{d_k} \le \frac{\bar{\mathcal{C}}n}{s} \frac{U_g^{1-\gamma}}{c_2}.
  \end{equation*}
  Meanwhile, if $\gamma>1$, it follows from \Cref{lem:upper_bound_of_d_k_1} and $\norm{g_k}\ge \eps$ that 
  \begin{equation*}
    \norm{d_k} \le \frac{\bar{\mathcal{C}}n}{s} \frac{1}{c_2\eps^{\gamma-1}}.
  \end{equation*}
This completes the proof.
\endproof

When $\norm{g_k}\ge \eps$, we have from \Cref{lem:upper_bound_of_d_k_2} that 
\begin{equation*}
  x_k + \tau d_k \in \Omega + B(0, r(\eps)), \ \forall \tau\in[0,1].
\end{equation*}
By boundedness of $\Omega + B(0, r(\eps))$ and by using the fact that $f$ is twice continuously differentiable, we deduce that there exists $U_H(\eps)>0$ such that
\begin{equation}\label{eq:upper_bound_of_hessian}
  \norm{\nabla^2 f(x)} \le U_H(\eps), \ \forall x \in \Omega + B(0,r(\eps)).
\end{equation}

The following lemma indicates that a step size smaller than some constant satisfies Armijo's rule when $\norm{g_k}\ge \eps$.

\begin{lemma}\label{lem:lower_bound_step_size}
  Suppose that \Cref{assumption:RS-RNM_global} holds. Suppose also that there exists $\eps>0$ such that $\norm{g_k}\ge \eps$.
  Then, with probability at least $1-2e^{-s}$, a step size $t_k'>0$ such that 
  \begin{equation*}
    t_k' \le \frac{2(1-\alpha)c_2^2 \eps^{2\gamma}s}{((1+c_1) \frac{\bar{\mathcal{C}}n}{s} U_H(\eps) + c_2U_g^\gamma) U_H(\eps)\bar{\mathcal{C}}n}
  \end{equation*}
  satisfies Armijo's rule, i.e., 
  \begin{equation*}
    f(x_k) - f(x_k + t_k' d_k) \ge -\alpha t_k' g_k^\T d_k.
  \end{equation*}
\end{lemma}
\proof{Proof.}
  From Taylor's theorem, there exists $\tau_k' \in (0,1)$ such that
  \begin{equation*}
    f(x_k+t_k'd_k) = f(x_k) + t_k'g_k^\T d_k + \frac 12 {t_k'}^2 d_k^\T \nabla^2 f(x_k+\tau_k' t_k 'd_k)d_k.
  \end{equation*}
  Then, we have
 
 \begin{align}
	&f(x_k) - f(x_k+t_k'd_k) + \alpha t_k' g_k^\T d_k \nonumber \\
		=&(\alpha-1)t_k' g_k^\T d_k - \frac 12 {t_k'}^2 d_k^\T \nabla^2 f(x_k+\tau_k' t_k' d_k)d_k \nonumber\\
			=&(1-\alpha)t_k' g_k^\T P_k^\T M_k^{-1} P_k g_k - \frac 12 {t_k'}^2 g_k^\T P_k^\T M_k^{-1} P_k \nabla^2 f(x_k+\tau_k' t_k'd_k) P_k^\T M_k^{-1} P_kg_k \label{eq:armijo_by_taylor}\\
				& \hspace{29em} \text{(by \eqref{eq:direction})} \nonumber\\
			\ge& (1-\alpha)t_k' \lambda_{\min}(M_k^{-1}) \norm{P_k g_k}^2 \nonumber\\
			& \phantom{(1-\alpha)t_k' } - \frac 12 {t_k'}^2\lambda_{\max}(\nabla^2 f(x_k+\tau_k' t_k' d_k)) \lambda_{\max} (M_k^{-1} P_k  P_k^\T M_k^{-1}) \norm{P_kg_k}^2 \nonumber\\
			\ge& (1-\alpha)t_k' \lambda_{\min}(M_k^{-1}) \norm{P_k g_k}^2 - \frac 12 {t_k'}^2U_H(\eps) \lambda_{\max} (M_k^{-1} P_k  P_k^\T M_k^{-1}) \norm{P_kg_k}^2, \nonumber\\
			&\hspace{29em} \text{(by \eqref{eq:upper_bound_of_hessian})} \nonumber
	\end{align}  
	
	where the first inequality derives from the fact that
	\begin{align*}
		&g_k^\T P_k^\T M_k^{-1} P_k \nabla^2 f(x_k+\tau_k' t_k'd_k) P_k^\T M_k^{-1} P_kg_k \le \lambda_{\max}( M_k^{-1} P_k \nabla^2 f(x_k+\tau_k' t_k'd_k) P_k^\T M_k^{-1})\|P_kg_k\|^2 \\
		&\le \lambda_{\max}(\nabla^2 f(x_k+\tau_k' t_k' d_k)) \lambda_{\max} (M_k^{-1} P_k  P_k^\T M_k^{-1}) \norm{P_kg_k}^2.
	\end{align*}
	By \Cref{lem:concentration_of_PP}, we have that, with probability at least $1-2e^{-s}$, $\norm{P_k P_k^\T} \le \frac{\bar{\mathcal{C}}n}{s}$. 
In addition, we have $\norm{H_k} \le U_H(\eps)$ from \cref{eq:upper_bound_of_hessian}, which gives us $\norm{P_k H_k P_k^\T} \le \frac{\bar{\mathcal{C}}n}{s} U_H(\eps)$. 
 For these reasons, we obtain evaluation of the values of $\lambda_{\min}(M_k^{-1})$ and $\lambda_{\max}(M_k^{-1}P_kP_k^\T M_k^{-1})$:
  \begin{align}
    \lambda_{\min}(M_k^{-1}) &= \frac{1}{\lambda_{\max}(M_k)} \nonumber\\
                             &= \frac{1}{\lambda_{\max}(P_k H_k P_k^\T + c_1\Lambda_k I_s + c_2\norm{g_k}^\gamma I_s)}\nonumber\\
                           &\ge \frac{1}{\frac{\bar{\mathcal{C}}n}{s} U_H(\eps) + c_1 \frac{\bar{\mathcal{C}}n}{s} U_H(\eps) + c_2\norm{g_k}^\gamma}, \label{eq:lower_bound_M_inv}\\[0.7em]
    \lambda_{\max}(M_k^{-1}P_kP_k^\T M_k^{-1}) &\le \norm{P_kP_k^\T}\lambda_{\max}(M_k^{-1})^2  \nonumber \\
                           &\le \frac{\bar{\mathcal{C}}n}{s}\frac{1}{\lambda_{\min}(P_k H_k P_k^\T + c_1\Lambda_k I_s + c_2\norm{g_k}^\gamma I_s)^2} \nonumber\\
                           &\le \frac{\bar{\mathcal{C}}n}{s} \frac{1}{c_2^2 \norm{g_k}^{2\gamma}},\nonumber
  \end{align}
  so that we have
  \begin{align*}
    &f(x_k) - f(x_k+t_k'd_k) + \alpha t_k' g_k^\T d_k\\
    \ge& \frac{(1-\alpha)t_k'}{\frac{\bar{\mathcal{C}}n}{s} U_H(\eps) + c_1 \frac{\bar{\mathcal{C}}n}{s} U_H(\eps) + c_2\norm{g_k}^\gamma} \norm{P_k g_k}^2 - \frac 12 {t_k'}^2 \frac{\bar{\mathcal{C}}n}{s} \frac{U_H(\eps)}{c_2^2 \norm{g_k}^{2\gamma}} \norm{P_kg_k}^2\\
    \ge& \frac{(1-\alpha)t_k'}{\frac{\bar{\mathcal{C}}n}{s} U_H(\eps) + c_1 \frac{\bar{\mathcal{C}}n}{s} U_H(\eps) + c_2U_g^\gamma} \norm{P_k g_k}^2 - \frac 12 {t_k'}^2 \frac{\bar{\mathcal{C}}n}{s} \frac{U_H(\eps)}{c_2^2 \eps^{2\gamma}} \norm{P_kg_k}^2\\
    & \hspace{23em} \text{(by \eqref{eq:upper_bound_g_k} and $\norm{g_k}\ge \eps$)}\\
    =& \frac{\bar{\mathcal{C}}U_H(\eps) n }{2c_2^2 \eps^{2\gamma}s}t_k' \left( \frac{2(1-\alpha)c_2^2 \eps^{2\gamma}s}{((1+c_1) \frac{\bar{\mathcal{C}}n}{s} U_H(\eps) + c_2U_g^\gamma) U_H(\eps)\bar{\mathcal{C}}n} - t_k'\right) \norm{P_k g_k}^2\\
    \ge& 0,
  \end{align*}
  which completes the proof.
\endproof

As a consequence of this lemma, it turns out that the step size $t_k$ used in RS-RNM can be bounded from below by some constant.
\begin{corollary}\label{cor:lower_bound_t_k}
  Suppose that \Cref{assumption:RS-RNM_global} holds. Suppose also that there exists $\eps>0$ such that $\norm{g_k}\ge \eps$.
  Then, with probability at least $1-2e^{-s}$, the step size $t_k$ chosen in Line~6 of RS-RNM satisfies
  \begin{equation} \label{eq:lower_bound_t_k}
    t_k \ge t_{\min} (\eps),
  \end{equation}
  where
  \begin{equation*}
    t_{\min} (\eps) = \min\left(1, \frac{2(1-\alpha)\beta c_2^2 \eps^{2\gamma}s}{((1+c_1) \frac{\bar{\mathcal{C}}n}{s} U_H(\eps) + c_2U_g^\gamma) U_H(\eps)\bar{\mathcal{C}}n}\right).
  \end{equation*}
\end{corollary}
\proof{Proof.}
  If
  \begin{equation*}
    \frac{2(1-\alpha)c_2^2 \eps^{2\gamma}s}{((1+c_1) \frac{\bar{\mathcal{C}}n}{s} U_H(\eps) + c_2U_g^\gamma) U_H(\eps)\bar{\mathcal{C}}n} > 1,
  \end{equation*}
  we know that $t_k=1$ satisfies Armijo's rule \cref{eq:Armijo_rule} from \Cref{lem:lower_bound_step_size}. If not, there exists $l_k\in\{0,1,2,\ldots\}$ such that 
  \begin{equation*}
    \beta^{l_k + 1} < \frac{2(1-\alpha)c_2^2 \eps^{2\gamma}s}{((1+c_1) \frac{\bar{\mathcal{C}}n}{s} U_H(\eps) + c_2U_g^\gamma) U_H(\eps)\bar{\mathcal{C}}n} \le \beta^{l_k}, 
  \end{equation*}
  and by \Cref{lem:lower_bound_step_size}, we have that the step size $\beta^{l_k + 1}$ satisfies Armijo's rule \cref{eq:Armijo_rule}.
  Then, from the definition of $\beta^{l_k}$ in Line~6 of RS-RNM, we have 
  \begin{equation*}
    t_k = \beta^{l_k} \ge \beta^{l_k + 1} = \beta\cdot \beta^{l_k} \ge \frac{2(1-\alpha)\beta c_2^2 \eps^{2\gamma}s}{((1+c_1) \frac{\bar{\mathcal{C}}n}{s} U_H(\eps) + c_2U_g^\gamma) U_H(\eps)\bar{\mathcal{C}}n}.
  \end{equation*}
  This completes the proof.
\endproof

Using Corollary \ref{cor:lower_bound_t_k}, we can show the global convergence of RS-RNM under \Cref{assumption:RS-RNM_global}.
\begin{theorem}\label{thm:global_convergence}
  Suppose that \Cref{assumption:RS-RNM_global} holds. Let $\delta\in (0,1)$ and define $\delta_s:=2\left(\exp(-\frac{\C_0}{4}s)+ \exp(-s)\right)$ and 
  \begin{equation*}
    m=\left\lfloor \frac{f(x_0)-f^*}{(1-\delta )(1-\delta_s)p(\eps)\eps^2}\right\rfloor +1,
    \quad\text{where}\quad
    p(\eps) = \frac{\alpha t_{\min}(\eps)}{2\bar{\mathcal{C}}(1+c_1) \frac{n}{s} U_H(\eps) + 2c_2 U_g^\gamma}. 
  \end{equation*}
  Then, with probability at least $1-\exp\left(-\frac{\delta^2}{2}(1-\delta_s)m\right)$ there exists 
  $k\in\{0,1,\dots,m-1\}$ such that $\norm{g_k} < \eps$. 
\end{theorem}
\proof{Proof.}
  We first notice that, by \Cref{lem:JLL_random_matrix}, applied with $\varepsilon=1/2$, and \Cref{lem:concentration_of_PP}, 
  we have, using \cref{eq:unionbound}, that $\norm{P_kg_k}^2\ge \frac{1}{2}\norm{g_k}^2$ and $\|P_kP_k^\top\|\le \bar{\mathcal{C}}\frac{n}{s}$ holds for all $k\in\{0,1,\ldots,m-1\}$ with the given probability. 

  Suppose, for the sake of contradiction, that $\norm{g_k} \ge \eps$ for all $k\in\{0,1,\dots,m-1\}$.
  From Armijo's rule \cref{eq:Armijo_rule}, we can estimate how much the function value decreases in one iteration. We have that with probability  $1- 2\left(\exp(-\frac{\C_0}{4}s)+ \exp(-s)\right)$:
  \begin{align*}
 	f(x_k) - f(x_{k+1}) &\ge -\alpha t_k g_k^\T d_k\\
 	&=   \alpha t_k g_k^\T P_k^\T M_k^{-1} P_k g_k\\
 	&\ge \alpha t_k \lambda_{\min(M_k^{-1})} \norm{P_k g_k}^2\\
 	&\ge \frac{\alpha t_{\min}(\eps)}{2(1+c_1) \frac{\bar{\mathcal{C}}n}{s} U_H(\eps) + 2c_2\norm{g_k}^\gamma}\norm{g_k}^2 \\
 	&\hspace{10em} \text{(by $\norm{P_kg_k}^2\ge \tfrac{1}{2}\norm{g_k}^2$ )}\\
 	&\ge p(\eps)\eps^2. \hspace{5.75em}  \text{(by \eqref{eq:upper_bound_g_k} and $\norm{g_k}\ge \eps$)}
 \end{align*} 
  Let us denote by $\mathcal{A}_k$ the event, only depending of $P_k$, where the above inequality holds. Conditionally to the complement of $\mathcal{A}_k$ we have only that $f(x_k) - f(x_{k+1})\ge 0$.
  Let us denote by $T_k\in \{0,1\}$ the random variable equal to $1$ if and only if $\mathcal{A}_k$ holds.
  Notice that the random variables $\{T_k\}$ are mutually independent because $T_k$ depends only on $P_k$. By the above remark we have that for all $k$: $f(x_k) - f(x_{k+1}) \ge p(\eps)\eps^2T_k$. Hence
  by adding up all these inequalities from $k=0$ to $k=m-1$, we get
  \begin{equation}
    f(x_0)-f(x_m) \ge p(\eps)\eps^2\sum\limits_{k=0}^{m-1}T_k.
    \label{eq:sumT_RHS}
  \end{equation}
  Since, for all $k$, $\mathbb{E}[T_k]\ge 1- 2\left(\exp(-\frac{\C_0}{4}s)+ \exp(-s)\right):=1-\delta_s $,
  we have by a Chernoff bound (see \cite{vershynin2018high}), that for all $\delta \in (0,1)$,
 \begin{equation}\label{eq:Chernoff_bound}
  \mathbb{P}\left(\sum\limits_{k=0}^{m-1}T_k \ge (1-\delta )(1-\delta_s)m\right) \ge 1-\exp\left(-\frac{\delta^2}{2}(1-\delta_s)m\right).
  \end{equation}
  Notice that by definition of $m$, we have that
  $$m >  \frac{f(x_0)-f^*}{(1-\delta )(1-\delta_s)p(\eps)\eps^2}.$$
  Hence
  \begin{equation}\label{eq:auxa}
  	(1-\delta )(1-\delta_s) p(\eps)\eps^2 m > f(x_0)-f^*.
  \end{equation}
  Thus, with probability at least  $1-\exp\left(-\frac{\delta^2}{2}(1-\delta_s)m\right)$
  \begin{align*}
    f(x_0)-f^* \ge& f(x_0)-f(x_m) \\
    \ge& (1-\delta )(1-\delta_s) mp(\eps)\eps^2 \\
    >  &f(x_0)-f^*,
  \end{align*}
  where the second inequality holds by \eqref{eq:sumT_RHS} together with \eqref{eq:Chernoff_bound} and the strict inequality holds by \cref{eq:auxa}.
  This is a contradiction, hence there exists  $k\in\{0,1,\dots,m-1\}$ such that $\norm{g_k} < \eps$.
\endproof

Because of the dependency of $ p(\eps) $ on $\eps$, the above discussion can not lead to the iteration complexity analysis, as we need to quantify the exact dependency of the
iteration complexity bound with respect to $\varepsilon$. This will be done, under a few additional assumptions, in the next subsection.

\subsection{Global iteration complexity}\label{sec:RSRNM_global_complexity}
We now estimate the global iteration complexity of the RS-RNM under \Cref{assumption:RS-RNM_global} and the following assumption.

\begin{assumption}\label{assumption:global_iteration_complexity}
  \ \\
  \vspace{-1em}
  \begin{enumerate}
    \renewcommand{\labelenumi}{(\roman{enumi})} 
    \item $\gamma \le 1/2$,
    \item $\alpha \le 1/2$,
    \item There exists $L_H>0$ such that 
    \begin{equation*}
      \norm{\nabla^2f(x)-\nabla^2f(y)} \le L_H \norm{x-y}, \quad \forall x, y \in \Omega + B(0, r_1),
    \end{equation*}
    where $r_1 := \dfrac{\bar{\mathcal{C}}U_g^{1-\gamma} n}{c_2 s}$. 
  \end{enumerate}
\end{assumption}

    From the definition of $r_1$ in $(iii)$, 
     \Cref{lem:upper_bound_of_d_k_1} and \cref{eq:upper_bound_g_k}, we have
\begin{equation*}
  \norm{d_k} \le \frac{\bar{\mathcal{C}}n}{s}\frac{\norm{g_k}^{1-\gamma}}{c_2} \le \frac{\bar{\mathcal{C}}n}{s}\frac{U_g^{1-\gamma}}{c_2} = r_1.
\end{equation*}
Note that unlike \cref{eq:upper_bound_d_k_2}, the bound has no dependency on $\eps$.
For this reason, we have 
\begin{equation*}
  x_k + \tau d_k \in \Omega + B(0, r_1), \ \forall \tau\in[0,1].
\end{equation*}
Moreover, since $\Omega + B(0, r_1)$ is bounded and $f$ is twice continuously differentiable, there exists $U_H>0$ such that
\begin{equation}\label{eq:upper_bound_of_hessian_2}
  \norm{\nabla^2 f(x)} \le U_H, \ \forall x \in \Omega + B(0,r_1).
\end{equation}

Similar to the result of \Cref{lem:lower_bound_step_size}, we can show that a step size smaller than some constant satisfies Armijo's rule and therefore, $t_k$ can be bounded from below by some constant.

\begin{lemma}\label{lem:lower_bound_stepsize_2}
Suppose that \Cref{assumption:RS-RNM_global} and \Cref{assumption:global_iteration_complexity} hold. 
Then,  with probability at least $1-2e^{-s}$, a step size $t_k'>0$ such that 
  \begin{equation*}
    t_k' \le \min\left(1,\frac{c_2^2s^2}{\bar{\mathcal{C}}^2L_H U_g^{1-2\gamma} n^2}\right),
  \end{equation*}
  satisfies Armijo's rule, i.e., 
  \begin{equation*}
    f(x_k) - f(x_k + t_k' d_k) \ge -\alpha t_k' g_k^\T d_k.
  \end{equation*}
\end{lemma}
\proof{Proof.}
  As \cref{eq:armijo_by_taylor} is obtained in the proof of \Cref{lem:lower_bound_step_size}, there exists $\tau_k'\in(0,1)$ such that
  \begin{align*}
    &f(x_k) - f(x_k+t_k'd_k) + \alpha t_k' g_k^\T d_k\\
   =&(1-\alpha)t_k' g_k^\T P_k^\T M_k^{-1} P_k g_k - \frac 12 {t_k'}^2 g_k^\T P_k^\T M_k^{-1} P_k \nabla^2 f(x_k+\tau_k't_k'd_k) P_k^\T M_k^{-1} P_kg_k.
  \end{align*}
  Since we have $1-\alpha \ge 1/2 \ge t'_k/2$ from \Cref{assumption:global_iteration_complexity} $(ii)$, we obtain
  \begin{align}
   & f(x_k) - f(x_k+t_k'd_k) + \alpha t_k' g_k^\T d_k \nonumber \\
   \ge& \frac{1}{2}{t_k'}^2 g_k^\T P_k^\T M_k^{-1} P_k g_k - \frac 12 {t_k'}^2 g_k^\T P_k^\T M_k^{-1} P_k \nabla^2 f(x_k+\tau_k't_k'd_k) P_k^\T M_k^{-1} P_kg_k \nonumber\\
   =&\frac{1}{2}{t_k'}^2 g_k^\T P_k^\T (M_k^{-1} - M_k^{-1}P_k H_k P_k^\T M_k^{-1}) P_k g_k \nonumber\\
   & \phantom{\frac{1}{2}{t_k'}^2 g_k^\T P_k^\T} -\frac 12 {t_k'}^2 g_k^\T P_k^\T M_k^{-1} P_k (\nabla^2 f(x_k+\tau_k't_k'd_k)-H_k) P_k^\T M_k^{-1} P_kg_k.\label{eq:armijo_evaluate}
  \end{align}
  We next evaluate the first and second terms respectively. Since we have
  \begin{align}
    M_k^{-1} - M_k^{-1}P_k H_k P_k^\T M_k^{-1} &= M_k^{-1} - M_k^{-1}(M_k- 
    \eta_k I_s) M_k^{-1}\nonumber\\
    &= 
    \eta_k (M_k^{-1})^2,\label{eq:M_k_inv-}
  \end{align}
  the first term can be bounded as follows:
  \begin{align*}
    \frac{1}{2}{t_k'}^2 g_k^\T P_k^\T (M_k^{-1} - M_k^{-1}P_k H_k P_k^\T M_k^{-1}) P_k g_k &= \frac{1}{2}{t_k'}^2
    \eta_k\norm{M_k^{-1}P_kg_k}^2\\
    &\ge  \frac{1}{2}{t_k'}^2 c_2\norm{g_k}^\gamma\norm{M_k^{-1}P_kg_k}^2.
  \end{align*}
  Using \Cref{lem:concentration_of_PP} and \Cref{assumption:global_iteration_complexity} $(iii)$, we also obtain, with probability at least $1-2e^{-s}$, the bound of the second term:
  \begin{align*}
    &\frac 12 {t_k'}^2 g_k^\T P_k^\T M_k^{-1} P_k (\nabla^2 f(x_k+\tau_k't_k'd_k)-H_k) P_k^\T M_k^{-1} P_kg_k \\
    \le& \frac 12 {t_k'}^2 \norm{\nabla^2 f(x_k+\tau_k't_k'd_k)-H_k} \norm{P_k P_k^\T} \norm{M_k^{-1} P_kg_k}^2\\
    \le& \frac{\bar{\mathcal{C}}n}{2s} L_H {t_k'}^3  \norm{d_k}\norm{M_k^{-1} P_kg_k}^2.
  \end{align*}
  Thus, we have
  \begin{align} \label{eq:localt_k}
    f(x_k) - f(x_k+t_k'd_k) + \alpha t_k' g_k^\T d_k
    &\ge \frac{1}{2}{t_k'}^2 \left(c_2\norm{g_k}^\gamma - \frac{\bar{\mathcal{C}}n}{s} L_H t_k' \norm{d_k} \right) \norm{M_k^{-1} P_kg_k}^2 \notag\\
    &= \frac{\bar{\mathcal{C}}n}{2s} L_H {t_k'}^2\norm{d_k} \left(\frac{c_2s\norm{g_k}^\gamma}{\bar{\mathcal{C}}L_H n\norm{d_k}} - t_k'  \right) \norm{M_k^{-1} P_kg_k}^2.
  \end{align}
  Moreover, from \cref{eq:upper_bound_g_k}, \Cref{lem:upper_bound_of_d_k_1} and \Cref{assumption:global_iteration_complexity} $(i)$, we have
  \begin{equation*}
    \frac{\norm{g_k}^\gamma}{\norm{d_k}} \ge \frac{c_2s}{\bar{\mathcal{C}}n \norm{g_k}^{1-2\gamma}} \ge \frac{c_2s}{\bar{\mathcal{C}}U_g^{1-2\gamma}n},
  \end{equation*}
  so that we finally obtain 
  \begin{align*}
    f(x_k) - f(x_k+t_k'd_k) + \alpha t_k' g_k^\T d_k &\ge \frac{\bar{\mathcal{C}}n}{2s} L_H {t_k'}^2\norm{d_k} \left(\frac{c_2^2 s^2}{\bar{\mathcal{C}}^2L_H U_g^{1-2\gamma} n^2} - t_k'  \right) \norm{M_k^{-1} P_kg_k}^2\\
    &\ge 0.
  \end{align*}
  This completes the proof.
\endproof

\begin{corollary}\label{cor:lower_bound_t_k_2}
  Suppose that \Cref{assumption:RS-RNM_global} and \Cref{assumption:global_iteration_complexity} hold.
  Then, with probability at least $1-2e^{-s}$, the step size $t_k$ chosen in Line~6 of RS-RNM satisfies
  \begin{equation}\label{eq:lower_bound_t_k_2}
    t_k \ge t_{\min},
  \end{equation}
  where 
  \begin{equation*}
    t_{\min} = \min\left(1, \frac{\beta c_2^2s^2}{\bar{\mathcal{C}}^2L_H U_g^{1-2\gamma} n^2}\right).
  \end{equation*}
\end{corollary}
\proof{Proof.}
  We get the conclusion in the same way as in the proof of Corollary \ref{cor:lower_bound_t_k} using Lemma \ref{lem:lower_bound_stepsize_2}.
\endproof

\begin{remark}
  Since \cref{eq:lower_bound_t_k_2} is equivalent to $\beta^{l_k} \ge t_{\min}$, and moreover 
  $$l_k \le \log t_{\min} / \log \beta,$$  Corollary \ref{cor:lower_bound_t_k_2} tells us that the number of the backtracking steps is bounded above by some constant independent of $k$.
\end{remark}

Now, we can obtain the global iteration complexity of RS-RNM. 
\begin{theorem}\label{thm:global_complexity}
  Suppose that \Cref{assumption:RS-RNM_global} and \Cref{assumption:global_iteration_complexity} hold. Consider any $\delta \in (0,1)$.
  Let 
  \begin{equation*}
  	m=\left\lfloor \frac{f(x_0)-f^*}{(1-\delta)(1-\delta_s)p\eps^2}\right\rfloor +1,
  	\quad\text{where}\quad
  	 p = \frac{\alpha t_{\min}}{2\bar{\mathcal{C}}(1+c_1) \frac{n}{s} U_H + 2c_2 U_g^\gamma},
  \end{equation*}
  and where $\delta_s=2\left(\exp(-\frac{\C_0}{4}s) - \exp(-s)\right)$.
  Then, we have that
  \begin{equation*}
    \sqrt{\frac{f(x_0)-f^*}{mp}} \ge \min_{k=0,1,\dots,m-1} \norm{g_k}
  \end{equation*}
  holds with probability at least $1-\exp\left(-\frac{\delta^2}{2}(1-\delta_s)m\right)$.
\end{theorem}
\proof{Proof.}
  Replacing $U_H(\eps)$ and $t_{\min(\eps)}$ with $U_H$, in \cref{eq:upper_bound_of_hessian_2}, and $t_{\min}$ respectively in the argument in the proof of Theorem \ref{thm:global_convergence}, 
  we have 
  \begin{equation*}
    f(x_k) - f(x_{k+1}) \ge p\norm{g_k}^2 \quad (k=0,1,\ldots, m-1),
  \end{equation*}
  with the given probability. Therefore, by using the same notation as in the proof of Theorem \ref{thm:global_convergence}, we obtain:
  \begin{align*}
  	f(x_0)-f^* &\ge f(x_0)-f(x_m) \\
  	&= \sum_{k=0}^{m-1} (f(x_k)-f(x_{k+1})) \\
  	&\ge p \sum_{k=0}^{m-1} \norm{g_k}^2 T_k\\
  	& \ge p\left(\min_{k=0,1,\dots,m-1}\norm{g_k}^2\right)\sum_{k=0}^{m-1} T_k \\
  	&\ge (1-\delta)(1-\delta_s)mp \left(\min_{k=0,1,\dots,m-1}\norm{g_k}^2\right),
  \end{align*}
  where the last inequality holds with probability $1-\exp\left(-\frac{\delta^2}{2}(1-\delta_s)m\right)$ as shown in \eqref{eq:Chernoff_bound}.
  This prove the theorem.
\endproof
If we ignore the probability, Theorem \ref{thm:global_complexity} shows that we get $\norm{g_k}\le \eps$ after at most $O(\eps^{-2})$ iterations. 
This global complexity $O(\eps^{-2})$ is the same as that obtained in \cite{ueda2010convergence} for the regularized Newton method.
Notice that, by a cubic regularization, the R-ARC algorithm in \cite{Shao} achieves  $O(\varepsilon^{-3/2})$ to obtain a first order stationary point. 

\section{Local convergence}\label{sec:local1}
In this section, we investigate local convergence properties of the sequence $\{x_k\}$ assuming that it converges to a  strict local minimizer $\bar{x}$. First we will show that the sequence converges locally linearly to the strict local minimizer; then we will prove that, when $f$ is strongly convex, we cannot aim at local super-linear convergence using random subspace. Finally, we will prove that when the Hessian at $\bar{x}$ is rank deficient then we can attain super-linear convergence for $s<n$ large enough.  

\begin{assumption}\label{assumption:local1}
	For all $x,y$
	$$\|\nabla^2 f(x) -\nabla^2 f(y)\| \le L_H\|x-y\|$$
	holds in some neighborhood $B_H$ of $\bar{x}$.
\end{assumption}

\subsection{Local linear convergence}
In this subsection we will show that the sequence $\{f(x_k)-f(\bar{x})\}$ converges locally linearly, i.e. there exists $\kappa \in (0,1)$ such that for $k$ large enough,
$$f(x_{k+1})-f(\bar{x}) \le (1-\kappa) (f(x_k)-f(\bar{x})).$$
{We will further prove that $\kappa$ can be expressed as $\kappa=O(\frac{s}{n \tilde{\kappa}(\nabla^2 f(\bar x))})$, where $\tilde{\kappa}(\nabla^2 f(\bar x))$ is the ratio of the largest eigenvalue value over the smallest non-zero eigenvalue of $\nabla^2 f(\bar{x})$.}
 Notice that, to the best of our knowledge, until now, local linear convergence is always proved for subspace algorithms assuming that the function is locally strongly convex or satisfies some PL-inequality \eqref{eq:pl}. 
 In this section we prove that under a H\"{o}lderian error bound condition, and an additional mild assumptions on the rank of the Hessian at the local minimizer, we can prove local linear convergence. 
 More precisely let us denote by $r=\rank(\nabla^2f(\bar{x}))$, which measures the number of positive eigenvalues of $\nabla^2 f(\bar{x})$. 
We will first prove, {under some assumption on the rank of the Hessian at $\bar{x}$ and on $s$,}  that for any $x$ in the a neighborhood of $\bar{x}$, the function
\begin{equation}\label{eq:ftilde}
	\tilde{f}_{x}: u\mapsto f(x+P^\top u),\quad \mbox{where $P$ is a random matrix sampled from $\mathcal{D}$} 
\end{equation}
is  strongly convex with high probability in a neighborhood of $0$. Let us fix $\sigma \in (0,1)$.  We recall here that $P\in \mathbb{R}^{s\times n}$ is equal to $\frac{1}{\sqrt{s}}$ times a random Gaussian matrix. In this subsection, we make the following additional assumptions:

\begin{assumption}\label{assumption:local2}
	\begin{enumerate}
		\renewcommand{\labelenumi}{(\roman{enumi})} 		
		\item There exists $\sigma \in (0,1)$ such that	$r=\rank(\nabla^2 f(\bar{x}))\ge \sigma n$.
		\item There exist $\rho \in (0,3) $ and $\tilde{C}$ such that in a neighborhood of $\bar{x}$, $f(x_k)-f(\bar{x})\ge \tilde{C}\|x_k-\bar{x}\|^\rho$ holds.
	\end{enumerate}
\end{assumption}

\begin{assumption}\label{assumption:s}
{We have that $s\le \min\left(\frac{\sigma}{4\mathcal{C}^2}, \frac{4(1-\sigma)}{\mathcal{C}^2}\right) n$.}
\end{assumption}

From \Cref{assumption:local2} $(i)$, $\nabla^2f(\bar{x})$ has $r$ positive eigenvalues, i.e,
$\lambda_1(\bar{x}) \ge \cdots \lambda_r(\bar{x}) >0$.
By continuity of the eigenvalues, there exists a neighborhood $\bar{B}$ of $\bar{x}$ such that for any $x \in \bar{B}$, $\lambda_r(x)\ge \frac{\lambda_r(\bar{x})}{2}$. Here, we assume, w.l.o.g. that $\bar{B}\subseteq B_H$, where $B_H$ is defined in \Cref{assumption:local1}.
Let us denote 
\begin{equation}\label{eq:lambda}
	\bar{\lambda}:=\frac{\lambda_r(\bar{x})}{2}.
\end{equation}
\Cref{assumption:local2} $(ii)$ is called a H\"{o}lderian growth condition or a H\"{o}lderian error bound condition \cite{holderian}. The condition is weaker than local strong convexity in the sense that it holds with $\rho = 2$ if $f$ is locally strongly convex. 

\begin{proposition}\label{prop:1}
	{
  Assume that \Cref{assumption:local2} $(i)$ and \Cref{assumption:s} hold. Let us consider $\tilde{f}_x$ defined by \eqref{eq:ftilde}. There exists a neighborhood $B^*\subseteq \bar{B}$ such that for any  $x \in B^*$,
	$$\nabla^2 \tilde{f}_{x}(0)\succeq \frac{n}{8s}\sigma\bar{\lambda} I_s$$
	holds with probability at least $1-6\exp(-s)$.}
\end{proposition}

\proof{Proof.}
	Let $x\in \bar{B}$ be fixed and let $P \in \mathbb{R}^{s \times n}$ be a Gaussian matrix. Because of $\nabla^2 \tilde{f}_{x}(0)= P \nabla^2 f(x)P^\top$,
        we have
         $u^\top \nabla^2 \tilde{f}_{x}(0)u =(P^\top u)^\top \nabla^2 f(x) (P^\top u)$ for any $u\in \mathbb{R}^s$. Let $\nabla^2 f(x)=U(x)D(x)U(x)^\top$ be the eigenvalue decomposition of $\nabla^2 f(x)$. Since $\nabla^2 \tilde{f}_{x}(0) =(P U(x)) D(x) (PU(x))^\top$ and  $PU(x)$ has the same distribution as $P$, we can assume here w.l.o.g. that $PU(x)=P$. Here 
	$$D(x)=\begin{pmatrix}
		\lambda_1(x) & 0 & \cdots & 0\\
		0 & \lambda_2(x) & \cdots &0 \\
		\vdots& & \ddots &  \vdots\\
		0& 0& \cdots & \lambda_n(x) 
	\end{pmatrix},$$
	where $\lambda_1(x)\ge \cdots \ge \lambda_n(x)$ and $\lambda_r(x)\ge \bar{\lambda}$ (since $x\in \bar{B}$).
     
	Let us decompose $P^\top$ such that
	$$P^\top =\begin{pmatrix}
		P^1 \\
		P^2
	\end{pmatrix}$$
	where $P^1 \in \mathbb{R}^{n_1 \times s}$ and $P^2\in \mathbb{R}^{n_2 \times s}$, where $n_1$ and $n_2$ are chosen such that $n_1=r$ and $n_2=n - r$. Furthermore let $D_1(x)$ and $D_2(x)$ be respectively the $n_1 \times n_1$ and $n_2 \times n_2$ diagonal matrix such that $D(x)=\begin{pmatrix}
		D_1(x) & 0 \\
		0& D_2(x)
	\end{pmatrix}$. We have 
	\begin{equation}\label{eq:decomposition}
		(P^\top u)^\top D(x) (P^\top u) = (P^1 u)^\top D_1(x) (P^1 u) + (P^2 u)^\top D_2(x) (P^2u).
	\end{equation}
	By \Cref{assumption:local2} $(i)$, and by definition of $\bar{B}$, we have that $D_1(x) \succeq \lambda_{r}(x) I_{n_1} \succeq \bar{\lambda} I_{n_1} \succ 0$, and $D_2(x) \succeq \lambda_n(x) I_{n_2}$. Hence from \cref{eq:decomposition}, we have
	\begin{equation}\label{eq:decomposition2}
		(P^\top u)^\top D(x) (P^\top u) \ge \bar{\lambda} \|P^1 u\|^2 + \lambda_n(x)\|P^2u\|^2 .
	\end{equation}
	Let $\sigma_{\max}(\cdot)$ and $\sigma_{\min}(\cdot)$ denote respectively the largest and the smallest singular value of a matrix.
	Using \cite[Theorem 4.6.1]{vershynin2018high}, there exists a constant $\mathcal{C}$ such that with probability at least $1-6\exp(-s)$:
	\begin{align}
		&\sqrt{\frac{n}{s}}-\mathcal{C} \le \sigma_{\min}(P^\top) \le \sigma_{\max}(P^\top) \le \sqrt{\frac{n}{s}}+\mathcal{C},  \label{eq:sv_PT} \\
		& \sqrt{\frac{n_1}{s}}-\mathcal{C} \le \sigma_{\min}(P^1) \le \sigma_{\max}(P^1) \le \sqrt{\frac{n_1}{s}}+\mathcal{C}, \nonumber \\
		& \sqrt{\frac{n_2}{s}}-\mathcal{C} \le \sigma_{\min}(P^2) \le \sigma_{\max}(P^2) \le \sqrt{\frac{n_2}{s}}+\mathcal{C}.\nonumber 
	\end{align}
More precisely, since all the three matrices $P^\top,P^1$ and $P^2$ are Gaussian random matrices, we can apply \cite[Theorem 4.6.1]{vershynin2018high} and deduce that each of the three inequalities above holds with probability $1-2\exp(-s)$. The probability that all the three equations hold is derived using \cref{eq:unionbound}.
	Hence, with probability at least $1-6e^{-s}$, for any $u \in \mathbb{R}^s$, 
	\begin{align*}
		\|P^1u\| \ge & \sqrt{n/s}\left(\frac{\sqrt{\frac{n_1}{s}}-\mathcal{C}}{\sqrt{n/s}}\right) \|u\|, \\
		\|P^2u\| \le & \sqrt{n/s}\left(\frac{\sqrt{\frac{n_2}{s}}+\mathcal{C}}{\sqrt{n/s}}\right) \|u\|.
	\end{align*}
	We have that $n_1\ge \sigma n$ and $n_2 \le (1-\sigma)n$. Furthermore, we have by \Cref{assumption:s} that {$s \le \frac{\sigma}{4\mathcal{C}^2}n$ implies that  $\sqrt{\frac{\sigma n}{s}}-\mathcal{C} \ge \frac{1}{2} \sqrt{\frac{\sigma n}{s}}$ and $s \le \frac{4(1-\sigma)}{4\mathcal{C}^2}n$ implies that $\sqrt{\frac{(1-\sigma) n}{s}}+\mathcal{C} \le 2\sqrt{\frac{(1-\sigma) n}{s}}$. Hence 
	$$ \frac{\sqrt{\frac{n_1}{s}}-\mathcal{C}}{\sqrt{n/s}}\ge \frac{1}{2} \sqrt{\sigma} \quad \& \quad \frac{\sqrt{\frac{n_2}{s}}+\mathcal{C}}{\sqrt{n/s}}\le 2\sqrt{(1-\sigma)} .$$}
	Therefore, 
	\begin{align*}
		\|P^1u\| \ge &{\frac{1}{2}}\sqrt{\sigma (n/s)} \|u\|, \\
		\|P^2u\| \le & {2}\sqrt{(1-\sigma) (n/s)} \|u\|.
	\end{align*}
	Hence, from \cref{eq:decomposition2}, we have that 
	$$ (P^\top u)^\top D(x) (P^\top u) \ge {n/s}\left(\frac{1}{{4}}\sigma\bar{\lambda} + {4}(1-\sigma)\min(\lambda_n(x),0)\right)\|u\|^2 .$$
	We conclude the proposition by noticing that $\min(\lambda_n(x),0)$ tends to $0$, hence the claim holds by considering a neighborhood $B^*\subseteq \bar{B}$ of $\bar{x}$ small enough.
\endproof
We deduce the following PL inequality for $\tilde{f}_x$ when $x\in B^*$.
\begin{proposition}\label{prop:2}
	{
 Assume that \Cref{assumption:local1}, \Cref{assumption:local2} $(i)$ and \Cref{assumption:s} hold, and let $P \in \mathbb{R}^{s \times n}$ be a Gaussian matrix. There exist neighborhoods $\hat{B}\subset B^*$ and $B_0$ (a neighborhood of $0 \in \mathbb{R}^s$) such that for any  $x \in \hat{B}$,
	$$ {\nabla\tilde{f}_x(0)^\top (P\nabla ^2f({x})P^\top)^{-1} \nabla \tilde{f}_x(0) \ge  f(x)-\min\limits_{u\in B_0} f(x+P^\top u)}$$
	holds with probability at least $1-6\exp(-s)$.}
\end{proposition}
\proof{Proof.}
	Let $\hat{B}\subset B^*$, and let $x\in \hat{B}$. By the Taylor expansion of $\tilde{f}_x$ at $0$, there exists $\tilde{x} \in [x,x+P^\top u]$ such that
	$$f(x+P^\top u) = f(x)+ (P\nabla f(x))^\top u +\frac{1}{2}u^\top P\nabla ^2f({\tilde{x}})P^\top u.$$
	Since, by Proposition \ref{prop:1}, we have that $P\nabla ^2f(\tilde{x})P^\top  \succ 0$ for any  $x+P^\top u \in B^*$, we deduce by \Cref{assumption:local1} that for $u$ small enough:
	\begin{equation}\label{eq:abc}
		f(x+P^\top u) \ge f(x)+ (P\nabla f(x))^\top u +\frac{1}{4}u^\top P\nabla ^2f({x})P^\top u.
	\end{equation}
	Let $B_0$ be a neighborhood of $0 \in \mathbb{R}^s$ such that, \eqref{eq:abc} holds, and $x+P^\top u \in B^*$ for any $x \in \hat{B}$. Let $g(u)=(P\nabla f(x))^\top u +\frac{1}{4}u^\top P\nabla ^2f({x})P^\top u$. By the above inequality we have that
	\begin{equation}\label{eq:subspace_min_ineq}
           \min\limits_{u \in B_0} f(x+P^\top u)\ge f(x)+ \min\limits_{u \in B_0} g(u).
       	\end{equation}
 By Proposition \ref{prop:1} we know that for any $u\in \mathbb{R}^s$ such that $x+P^\top u \in B^*$, $g$ is convex. Thus, the minimum is attained at the point $u^*$ satisfying 
	$$\nabla g(u^*)= P\nabla f(x)+\frac{1}{2}P\nabla ^2f({x})P^\top u^*=0.$$
	Hence, since $\|\nabla f(x)\|$ tends to $0$ as $x$ tends to $\bar{x}$, we can ensure, by taking $\hat{B}$ small enough, that $u^* \in B_0$.  Hence 
	\begin{align*}
		 \min\limits_{u \in B_0} g(u)&=		 -{2}(P\nabla f(x))^\top (P\nabla ^2f({x})P^\top)^{-1}P\nabla f(x) +\frac{1}{4}{4}(P\nabla f(x))^\top (P\nabla ^2f({x})P^\top)^{-1}P\nabla f(x)\\
		 &=-(P\nabla f(x))^\top (P\nabla ^2f({x})P^\top)^{-1}P\nabla f(x)
	\end{align*}
holds and \cref{eq:subspace_min_ineq} yields the desired inequality.
\endproof

Before proving local linear convergence, we prove the following technical proposition.
\begin{proposition}\label{prop:3}
  Assume that \Cref{assumption:local1}, \Cref{assumption:s}, and \Cref{assumption:local2} hold. 
	{There exists $k_0 \in \mathbb{N}$ such that if $k\ge k_0$,
	 we have with probability }$1-6(\exp(-s)+\exp(-\frac{\mathcal{C}_0}{4}s))$:
	$${ f(x_k)-\min\limits_{u \in B_0} \tilde{f}_{x_k}(u) \ge \frac{\lambda_0}{4\lambda_{\max}(\bar{H})\left(\sqrt{\frac{n}{s}}+\mathcal{C}\right)^2} (f(x_k)- f(\bar{x})),}$$
	where $\lambda_0$ is the minimal non-zero eigenvalue of $\bar{H}:=\nabla^2 f(\bar{x})$.
\end{proposition}
\proof{Proof.}
	Using a Taylor expansion around $\bar{x}$, we have that for all $y\in \hat{B}$,
	\begin{equation}\label{eq:3}
		|f(y) -f(\bar{x}) - \frac{1}{2}(y-\bar{x})^\top \bar{H} (y-\bar{x})| \le L_H\|y-\bar{x}\|^3, 
	\end{equation}
        where we define
        \begin{equation}\label{eq:barh}
	\bar{H}:=\nabla^2f(\bar{x}).
        \end{equation}
        Also, for $u\in \mathbb{R}^d$ small enough, we have by setting $y=x_k+P_k^\top u$ in \cref{eq:3}, that for $k$ large enough such that $x_k+P_k^\top u \in \hat{B}$,
	\begin{align}\label{eq:prop2a}
		&|f(x_k+P_k^\top u) -f(\bar{x})- \frac{1}{2}(x_k-\bar{x})^\top \bar{H} (x_k-\bar{x})- \frac{1}{2}u^\top P_k\bar{H}P_k^\top u- (P_k\bar{H}(x_k-\bar{x}))^\top u| \\
		&\le L_H\|x_k-\bar{x}+P_k^\top u\|^3 \nonumber
	\end{align}
	holds.
	
	Let $g(u)=\frac{1}{2}u^\top P_k\bar{H}P_k^\top u+ (P_k\bar{H}(x_k-\bar{x}))^\top u$. By a reasoning similar to that of Proposition \ref{prop:1}, $g$ is strongly convex with probability $1-6e^{-s}$ and hence is minimized at
        \begin{equation}\label{eq:opt_u}
        u^*=-(P_k\bar{H}P_k^\top)^{-1}P_k\bar{H}(x_k-\bar{x}).
        \end{equation}
    Notice that as $k$ tends to infinity $\|u^*\|$ tends to $0$, hence for $k$ large enough we have $x_k+P_k^\top u^* \in \hat{B}$ and $u^*\in B_0$.
        Plugging \cref{eq:opt_u} in \cref{eq:prop2a} yields
	\begin{align*}
		&f(x_k+P_k^\top u^*)\le f(\bar{x})+ \frac{1}{2}(x_k-\bar{x})^\top \bar{H} (x_k-\bar{x}) - \\
		& \frac{1}{2}(x_k-\bar{x})^\top \bar{H}P_k^\top(P_k\bar{H}P_k^\top)^{-1}P_k\bar{H}(x_k-\bar{x})+L_H\|x_k-\bar{x}+P_k^\top u^*\|^3,
	\end{align*} 
	from which we deduce
	\begin{align*}
		&f(x_k)-	f(x_k+P_k^\top u^*) \ge \\
		& f(x_k)-f(\bar{x})-\frac{1}{2}(x_k-\bar{x})^\top \bar{H} (x_k-\bar{x})+\frac{1}{2}(x_k-\bar{x})^\top \Pi(x_k-\bar{x}) - L_H\|x_k-\bar{x}+P_k^\top u^*\|^3,
	\end{align*}
	where $\Pi=\bar{H}P_k^\top(P_k\bar{H}P_k^\top)^{-1}P_k\bar{H}$. Using \cref{eq:3}, we further obtain
	\begin{equation}\label{eq:4}
		f(x_k)-	f(x_k+P_k^\top u^*) \ge \frac{1}{2} (x_k-\bar{x})^\top \Pi(x_k-\bar{x}) - L_H(\|x_k-\bar{x}+P_k^\top u^*\|^3+\|x_k-\bar{x}\|^3).
	\end{equation}

	We have $(x_k-\bar{x})^\top \Pi(x_k-\bar{x})=(\bar{H}^{1/2}(x_k-\bar{x}))^\top \bar{\Pi}(\bar{H}^{1/2}(x_k-\bar{x}))$,
	where $\bar{\Pi}:=\bar{H}^{1/2}P_k^\top(P_k\bar{H}P_k^\top)^{-1}P_k\bar{H}^{1/2}$ is an orthogonal projection matrix into $\Range(\bar{H}^{1/2}P_k^\top)$ parallel to $ \ker{P_k\bar{H}^{1/2}}$. Hence
	$$(x_k-\bar{x})^\top \Pi(x_k-\bar{x}) =\|\bar{\Pi}\bar{H}^{1/2}(x_k-\bar{x})\|^2. $$ 
	Since $\|P_k \bar{H}^{1/2}\|^2\|\bar{\Pi}\bar{H}^{1/2}(x_k-\bar{x})\|^2\ge \|P_k \bar{H}^{1/2}\bar{\Pi}\bar{H}^{1/2}(x_k-\bar{x})\|^2$,
	we have
	\begin{align}\label{eq:align}
			(x_k-\bar{x})^\top \Pi(x_k-\bar{x}) &\ge \frac{1}{\|P_k\bar{H}^{1/2}\|^2}\|P_k\bar{H}^{1/2}\bar{\Pi}\bar{H}^{1/2}(x_k-\bar{x})\|^2 \nonumber\\
			& =\frac{1}{\|P_k\bar{H}^{1/2}\|^2} \|P_k\bar{H}(x_k-\bar{x})\|^2 \nonumber\\
			& \ge \frac{1}{2\|P_k\bar{H}^{1/2}\|^2} \|\bar{H}(x_k-\bar{x})\|^2 \nonumber\\
			& \ge \frac{\lambda_0}{2\|P_k\bar{H}^{1/2}\|^2} \|\bar{H}^{1/2}(x_k-\bar{x})\|^2 \nonumber\\
			& = \frac{\lambda_0}{2\lambda_{\max}(P_k\bar{H}P_k)} \|\bar{H}^{1/2}(x_k-\bar{x})\|^2 \nonumber\\
			& =\frac{\lambda_0}{2\lambda_{\max}(P_k\bar{H}P_k)}(x_k-\bar{x})^\top \bar{H}(x_k-\bar{x}).
			\end{align}
		where the second inequality holds with probability at least $1-2\exp(-\frac{\mathcal{C}_0}{4}s)$ (by \Cref{lem:JLL_random_matrix} with $\eps=\frac{1}{2}$), and the third holds as $\lambda_0$ is the smallest non-zero eigenvalue of $\bar{H}$. The second equality holds as $\sigma_{\max}(P_k\bar{H}^{1/2})^2=\lambda_{\max} (P_k\bar{H}_kP_k)$.
		We have therefore proved that 
		\begin{equation}\label{eq:barpieigen}
			(\bar{H}^{1/2}(x_k-\bar{x}))^\top \bar{\Pi}(\bar{H}^{1/2}(x_k-\bar{x})) \ge \frac{\lambda_0}{2\lambda_{\max}(P_k\bar{H}P_k)} (x_k-\bar{x})^\top \bar{H}(x_k-\bar{x}).
		\end{equation}
	Hence, by \cref{eq:4}, we have 
	\begin{align}\label{eq:4a}
		f(x_k)-	f(x_k+P_k^\top u^*) \ge & \frac{\lambda_0}{4\lambda_{\max}(P_k\bar{H}P_k)} (x_k-\bar{x})^\top \bar{H}(x_k-\bar{x}) \nonumber\\
		& - L_H(\|x_k-\bar{x}+P_k^\top u^*\|^3+\|x_k-\bar{x}\|^3).
	\end{align}
	From \cref{eq:opt_u}, we have that $\|x_k-\bar{x}+P_k^\top u^*\|=\|(I_n-P_k^\top (P_k\bar{H}P_k^\top)^{-1}P_k\bar{H})(x_k-\bar{x})\|$. Hence
	\begin{equation}\label{eq:prop2ba}
		\|x_k-\bar{x}+P_k^\top u^*\| \le \|I_n-P_k^\top (P_k\bar{H}P_k^\top)^{-1}P_k\bar{H}\|\|x_k-\bar{x}\|. 
	\end{equation}
	Since $P_k^\top (P_k\bar{H}P_k^\top)^{-1}P_k\bar{H}$ is projection matrix (along $\Img(P_k^\top)$ parallel to $\Ker(P_kH)$), we have by \cite{andreev2014note} that
    \begin{equation}\label{eq:prop2b}
    	\|I_n-P_k^\top (P_k\bar{H}P_k^\top)^{-1}P_k\bar{H}\|=\|P_k^\top (P_k\bar{H}P_k^\top)^{-1}P_k\bar{H}\|.
    \end{equation}
	Furthermore, by Proposition \ref{prop:1}, we have that with probability at least $1-6\exp(-s)$, {
	$$P_k\bar{H}P_k^\top \succeq \frac{n}{8s}\sigma\bar{\lambda}I_s.$$
	Hence, we deduce from \cref{eq:prop2b} that
	\begin{equation}\label{eq:prop2c}
		\|I_n-P_k^\top (P_k\bar{H}P_k^\top)^{-1}P_k\bar{H}\| \le \frac{\|P_k^\top \|^2\|\bar{H}\| }{\frac{n}{8s}\sigma\bar{\lambda}}. 
	\end{equation}}
 Therefore, we deduce by \cref{eq:4a}, \cref{eq:prop2ba} and \cref{eq:prop2c} for $\beta_1>0$ suitably chosen, we have
	\begin{equation}\label{eq:prop2d}
		f(x_k)-	f(x_k+P_k^\top u^*) \ge \frac{\lambda_0}{4\lambda_{\max}(P_k\bar{H}P_k)} (x_k-\bar{x})^\top \bar{H}(x_k-\bar{x}) - \beta_1\|x_k-\bar{x}\|^3.
	\end{equation}
	By taking $y=x_k$ in \cref{eq:3}, we have that 
	$$ \frac{1}{2}(x_k-\bar{x})^\top \bar{H}(x_k-\bar{x}) \ge f(x_k)-f(\bar{x})-L_H\|x_k-\bar{x}\|^3.$$
	Hence, by \cref{eq:prop2d}
	$$f(x_k)-	f(x_k+P_k^\top u^*) \ge \frac{\lambda_0}{2\lambda_{\max}(P_k\bar{H}P_k)} (f(x_k)-f(\bar{x})) - (\frac{\lambda_0}{2\lambda_{\max}(P_k\bar{H}P_k)}L_H+\beta_1)\|x_k-\bar{x}\|^3.$$
	By \Cref{assumption:local2} $(ii)$,
	$$f(x_k)-	f(x_k+P_k^\top u^*) \ge (\frac{\lambda_0}{2\lambda_{\max}(P_k\bar{H}P_k)}- (\frac{\lambda_0}{2\lambda_{\max}(P_k\bar{H}P_k)}L_H+\beta_1)\frac{1}{\tilde{C}}\|x_k-\bar{x}\|^{3-\rho} ) (f(x_k)-f(\bar{x})). $$
	Since $\|x_k-\bar{x}\|$ tends to $0$ as $k$ tends to infinity and $\rho <3$, we have that for $k$ large enough
	$$ f(x_k)-\min\limits_{u\in B_0}f(x_k+P_k^\top u) \ge f(x_k)-	f(x_k+P_k^\top u^*)\ge  \frac{\lambda_0}{4\lambda_{\max}(P_k\bar{H}P_k)}(f(x_k)-f(\bar{x})),$$
	where the first inequality holds as, by \cref{eq:opt_u}, $u^*\in B_0$ for $k$ large enough. The probability bound in the statement of the theorem is obtained by using \cref{eq:unionbound}: in the whole proof we only use Lemma \ref{lem:JLL_random_matrix} with $\varepsilon=\frac{1}{2}$, which holds with probability at least $1-2\exp(-\frac{\mathcal{C}_0}{4}s)$, and the inequalities \cref{eq:sv_PT} which hold with probability at least $1-6\exp(-s)$. We also factorize the expression, using that $1-2\exp(-\frac{\mathcal{C}_0}{4}s)> 1-6\exp(-\frac{\mathcal{C}_0}{4}s) $.
	We end the proof by noticing that
	$\lambda_{\max}(P_k\bar{H}P_k) \le  \lambda_{\max}(\bar{H})\sigma_{\max}(P_k)^2 $, hence by the first equation of \cref{eq:sv_PT}
	\begin{equation}\label{eq:PHP}
		\lambda_{\max}(P_k\bar{H}P_k) \le \lambda_{\max}(\bar{H})\left(\sqrt{\frac{n}{s}}+\mathcal{C}\right)^2.
	\end{equation}
\endproof

We are now ready to prove the main theorem of this section.
\begin{theorem}\label{thm:linearconv}
	Assume that \Cref{assumption:local1}, \Cref{assumption:local2} and \Cref{assumption:s} hold. 
	There exist $0<\kappa <1$, $k_0 \in \mathbb{N}$, such that if $k\ge k_0$, then  
	$${f(x_{k+1})-f(\bar{x}) \le \left(1-\frac{1}{2}\alpha(1-\alpha)\frac{\lambda_0}{4\lambda_{\max}(\bar{H})\left(\sqrt{\frac{n}{s}}+\mathcal{C}\right)^2}\right) (f(x_k)-f(\bar{x}))}$$
	holds with probability at least $1-6(\exp(-s)+\exp(-\frac{\mathcal{C}_0}{4}s))$. Here $\alpha\in (0,1)$ is a parameter of Algorithm \ref{alg:RSN}.
\end{theorem}
\proof{Proof.}
We recall that we use a backtracking line search to find at each iteration $k$ a step-size $t_k$ such that
$$f(x_k +t_kd_k) \le f(x_k)+ \alpha t_k\nabla f(x_k)^\top d_k, $$
with $d_k=P_k^\top u_k$ and the update rule $t_k \leftarrow \beta t_k$ for $0<\alpha <1$  and $0<\beta<1$. We recall that 
\begin{equation}\label{eq:uk}
	u_k=-(P_k H_k P_k^\top+\eta_k I_s)^{-1}P_kg_k,
\end{equation} 
where we recall that $\eta_k=c_1\Lambda_k  + c_2\norm{g_k}^\gamma$. 
By a Taylor expansion of $f$ around $x_k$, there exists $x_k^* \in [x_k,x_{k+1}]$ such that
\begin{equation}\label{eq:taylor_f}
	f(x_k+t_kP_k^\top u_k) = f(x_k)+t_k (P_kg_k)^\top u_k+ \frac{t_k^2 }{2}u_k^\top P_k\nabla^2 f(x_k^*)P_k^\top u_k .
\end{equation}
Notice that  $\nabla^2 f$ is Lipschitz continuous (by \Cref{assumption:local1}). Furthermore, by Proposition \ref{prop:1}, for $k$ large enough, $P_k{H}_kP_k^\top$ is positive definite with probability at least $ 1-6\exp(-s)$ as the sequence $\{x_k\}$ converges to $\bar{x}$. Hence, for $k$ large enough
\begin{align*}
	u_k^\top P_k\nabla^2 f(x_k^*)P_k^\top  u_k  &\le u_k^\top P_k{H}_kP_k^\top  u_k + \|P_k^\top u_k\|^2  \|H_k-\nabla^2 f(x_k^*)\| \\
	&\le   u_k^\top P_k{H}_kP_k^\top u_k + L_H\|P_k^\top u_k\|^2  \|x_k-x_{k+1}\|\le   2u_k^\top P_k{H}_kP_k^\top u_k 
\end{align*}
holds with probability at least $1-6(\exp(-s)+\exp(-\frac{\mathcal{C}_0}{4}s))$. By \cref{eq:taylor_f}, we deduce that for $k$ large enough:
\begin{align*}		
	f(x_k+t_kP_k^\top u_k) &\le f(x_k)+t_k (P_kg_k)^\top u_k+ 2\frac{t_k^2 }{2}u_k^\top P_k{H}_kP_k^\top u_k \\
	&\le f(x_k)+t_k (P_kg_k)^\top u_k+ {t_k^2 }u_k^\top (P_k{H}_kP_k^\top+\eta_kI_s) u_k,
\end{align*}
where the second inequality holds as $\eta_k \ge 0$.
Let
\begin{equation}\label{eq:def_mu_squared}
  \mu_k^2 :=-g_k^\top d_k=(P_kg_k)^\top(P_kH_kP_k^\top+\eta_k I_s)^{-1}(P_kg_k).
\end{equation}
Since $(P_kg_k)^\top u_k=g_k^\top(P_k^\top u_k)=-\mu_k^2$, and by definition of $u_k$ in \cref{eq:uk}, we can write
\begin{equation}\label{eq:thm1}
	f(x_k+t_kP_k^\top u_k) \le f(x_k)-t_k \mu_k^2+  {t_k^2 }u_k^\top (P_k{H}_kP_k^\top+\eta_kI_s) u_k=f(x_k)-t_k \mu_k^2+  {t_k^2 }\mu_k^2.
\end{equation}
 Hence,  we have
$$ f(x_{k+1}) \le f(x_k)-t_k\left(1-t_k\right)\mu_k^2.$$
 Thus the step-size $t_k=1-\alpha$ satisfies the
 exit condition,
    $f(x_k) - f(x_k + t_k d_k) \ge -\alpha t_k g_k^\T d_k$,
 in the backtracking line search as we have
$$\left(1-t_k\right)=\alpha$$
 for such $t_k$. Therefore, the backtracking line search stops with some $t_k\ge 1-\alpha $, and 
 we have
\begin{equation}\label{eq:thm2}
	f(x_{k+1}) \le f(x_k)-\alpha  (1-\alpha) \mu_k^2.
\end{equation}

Notice that since $\eta_k$ tends to $0$, we have that
$$\mu_k^2=(P_kg_k)^\top(P_kH_kP_k^\top+\eta_k I_s)^{-1}(P_kg_k)\ge \frac{1}{2} (P_kg_k)^\top(P_k\bar{H}P_k^\top)^{-1}(P_kg_k)  .$$
Hence, by Proposition \ref{prop:2}, we have that when $k$ is large enough,
\begin{align}\label{eq:9}
	&f(x_{k+1})-f(\bar{x}) \le f(x_k)-f(\bar{x}) - \frac{1}{2}\alpha(1-\alpha)\left(f(x_k)-\min\limits_{u \in B_0} \tilde{f}_{x_k}(u)\right)
\end{align}
holds with probability at least $1-6(\exp(-s)+\exp(-\frac{\mathcal{C}_0}{4}s))$. By Proposition \ref{prop:3}, we have that 
 $ f(x_k)-\min\limits_{u \in B_0} \tilde{f}_{x_k}(u) \ge\frac{\lambda_0}{4\lambda_{\max}(\bar{H})\left(\sqrt{\frac{n}{s}}+\mathcal{C}\right)^2} (f(x_k)- f(\bar{x}))$ holds with probability at least  $1-6(\exp(-s)+\exp(-\frac{\mathcal{C}_0}{4}s))$. Hence

\begin{equation}\label{eq:10}
	f(x_{k+1})-f(\bar{x}) \le \left(1-\frac{1}{2}\alpha(1-\alpha)\frac{\lambda_0}{4\lambda_{\max}(\bar{H})\left(\sqrt{\frac{n}{s}}+\mathcal{C}\right)^2}\right)\left(f(x_k)-f(\bar{x})\right),
\end{equation}
which proves the theorem.
\endproof

\begin{remark}\label{rem:1}{	
Notice that the rate we obtain corresponds to a high probability estimation of the local convergence rate derived, when $f$ is assumed to be strongly convex, in the stochastic subspace cubic Newton method \cite{cubicNewton}. This can be seen in the proof of Proposition \ref{prop:3}, where the rate $\frac{\lambda_0}{4\lambda_{\max}(\bar{H})\left(\sqrt{\frac{n}{s}}+\mathcal{C}\right)^2}$ corresponds to a lower bound of $\lambda_{\min}(\bar{H}^{1/2}P_k^\top(P_k\bar{H}P_k^\top)^{-1}P_k\bar{H}^{1/2}) $, as seen in \cref{eq:barpieigen} and \cref{eq:PHP}. More specifically, this corresponds to a high probability lower bound of the parameter $\zeta=\lambda_{\min}[\mathbb{E}(\bar{\Pi})]=\lambda_{\min}[\mathbb{E}(\bar{H}^{1/2}P_k^\top(P_k\bar{H}P_k^\top)^{-1}P_k\bar{H}^{1/2})]$ that appears in the local convergence rate in Theorem 6.2 of \cite{cubicNewton}.}
\end{remark}

Let us define
$${\kappa :=\frac{1}{2}\alpha(1-\alpha)\frac{\lambda_0}{4\lambda_{\max}(\bar{H})\left(\sqrt{\frac{n}{s}}+\mathcal{C}\right)^2}<1 .}$$
We have the following direct corollary:
\begin{corollary}\label{cor:1}
	Assume that \Cref{assumption:local1}, \Cref{assumption:local2} and \Cref{assumption:s} hold. 
	There exist $k_0 \in \mathbb{N}$ such that if $k\ge k_0$, then, for any $m\in \mathbb{N}$, 
	$$f(x_{k+m})-f(\bar{x}) \le (1-\kappa)^m (f(x_k)-f(\bar{x}))$$
	holds with probability at least $1-6m(\exp(-s)+\exp(-\frac{\mathcal{C}_0}{4}s))$.
\end{corollary}
\proof{Proof.}
	This is a direct consequence of Theorem \ref{thm:linearconv} where the success probability is obtained by union bound, using \cref{eq:unionbound}.
\endproof
%

Notice that one can also derive an expectation version of Theorem \ref{thm:linearconv} as follows.
\begin{corollary}\label{cor:2}
{Assume that \Cref{assumption:local1}, \Cref{assumption:local2} and \Cref{assumption:s} hold. 
	There exist $k_0 \in \mathbb{N}$ such that if $k\ge k_0$, then, 
	$$\mathbb{E}\left[f(x_{k+1})-f(\bar{x})\right] \le (1- p^2\kappa) \mathbb{E}\left[f(x_{k})-f(\bar{x})\right],$$
	where $p:=1-6(\exp(-s)+\exp(-\frac{\mathcal{C}_0}{4}s))$. Here the expectation is taken with respect to the random variables $P_0,P_1,P_2,\cdots,P_k$.}
\end{corollary}
\proof{Proof.}
By \cref{eq:9} we have that 
	$$f(x_{k+1})-f(\bar{x}) \le f(x_k)-f(\bar{x}) - \frac{1}{2}\alpha(1-\alpha)\left(f(x_k)-\min\limits_{u \in B_0} \tilde{f}_{x_k}(u)\right)$$
	holds with probability $p=1-6(\exp(-s)+\exp(-\frac{\mathcal{C}_0}{4}s)) $.
	Let us denotes by $\mathcal{E}$ the event, with respect to $P_k,$ on which the above equation holds. Since $f(x_{k+1})-f(\bar{x}) \le f(x_k)-f(\bar{x})$ holds with probability one,  we can write that
	$$ f(x_{k+1})-f(\bar{x}) \le f(x_k)-f(\bar{x}) - \frac{1}{2}\alpha(1-\alpha)\left(f(x_k)-\min\limits_{u \in B_0} \tilde{f}_{x_k}(u)\right)\mathbf{1}_{\mathcal{E}},$$
	where $\mathbf{1}_{\mathcal{E}}$ is the indicator function over $\mathcal{E}$. Let us consider the following conditional expectation: $\mathbb{E}\left[ \cdot\ |\ P_0,...,P_{k-1} \right]$.  We have that 
	\begin{equation}\label{eq:cor2a}
		\mathbb{E}\left[ f(x_{k+1})-f(\bar{x})\ |\ P_0,...,P_{k-1} \right] \le f(x_k)-f(\bar{x}) - \frac{1}{2}\alpha(1-\alpha) \mathbb{E}\left[\left(f(x_k)-\min\limits_{u \in B_0} \tilde{f}_{x_k}(u)\right)\mathbf{1}_{\mathcal{E}}\ |\   P_0,...,P_{k-1}\right]
	\end{equation}
holds as $f(x_k)-f(\bar{x}) $ is measurable with respect to the sigma algebra generated by $P_1,\cdots,P_{k-1}$.
Let us define the event
		$$\mathcal{E}'=\left\{ f(x_k)-\min\limits_{u \in B_0} \tilde{f}_{x_k}(u) \ge \frac{\lambda_0}{4\lambda_{\max}(\bar{H})\left(\sqrt{\frac{n}{s}}+\mathcal{C}\right)^2} (f(x_k)- f(\bar{x}))\ |\ x_k\right\},$$
		on this sigma algebra, which holds by probability at least $p=1-6(\exp(-s)+\exp(-\frac{\mathcal{C}_0}{4}s)) $, by Proposition  \ref{prop:3}. By conditioning the right-hand-side of \cref{eq:cor2a} with respect to this event, we obtain that when $k$ is large enough
	\begin{align*}
		\mathbb{E}\left[f(x_{k+1})-f(\bar{x})\ |\ P_0,\cdots,P_{k-1}\right] &\le f(x_{k})-f(\bar{x}) - \frac{1}{2}\alpha(1-\alpha) \mathbb{E}\left[\frac{\lambda_0}{4\lambda_{\max}(\bar{H})\left(\sqrt{\frac{n}{s}}+\mathcal{C}\right)^2} (f(x_k)- f(\bar{x}))\mathbf{1}_{\mathcal{E}}\right] p \\
		&\le  (f(x_k)-f(\bar{x})) \left(1-\frac{1}{2}\alpha(1-\alpha) \frac{\lambda_0}{4\lambda_{\max}(\bar{H})\left(\sqrt{\frac{n}{s}}+\mathcal{C}\right)^2}p^2\right). 
	\end{align*}
Where the first inequality holds as in any case we have that $f(x_k)-\min\limits_{u \in B_0} \tilde{f}_{x_k}(u) \ge 0$.
	By taking the expectation with respect to $P_0,\cdots,P_{k-1}$ we deduce the corollary.
\endproof

Let consider the following assumption.
\begin{assumption}\label{ass:local2a}
{There exists $\rho>0$ such that for $k$ large enough
	\begin{equation}\label{eq:localer}
		\|\nabla f(x_k)\| \ge \rho \|x_k-\bar{x}\|.
	\end{equation}}
\end{assumption}
Notice that \Cref{ass:local2a} is actually stronger than  \Cref{assumption:local2}$(ii)$.
\begin{lemma}\label{lem:1}
	{We have, under \Cref{assumption:local1} and \Cref{ass:local2a}, that for $k$ large enough:
	$$ \frac{\rho}{2\sqrt{\lambda_{\max}(\bar{H})}} \|x_k-\bar{x}\| \le \|\sqrt{\bar{H}}(x_k-\bar{x})\| .$$}
\end{lemma}
\proof{Proof.}
	Using a Taylor expansion of $t \mapsto \nabla f(\bar{x}+t(x_k-\bar{x}))$ around $0$, we have that 
	\begin{equation}\label{eq:taylorgrad}
			\nabla f(x_k) = \nabla f(\bar{x}) + \int_{0}^{1} \nabla^2f(\bar{x}+t(x_k-\bar{x}))(x_k-\bar{x})dt=\int_{0}^{1} \nabla^2f(\bar{x}+t(x_k-\bar{x}))(x_k-\bar{x})dt.
	\end{equation}
	By \Cref{assumption:local1}, for any $t\in[0,1]$ we have $\|\nabla^2f(\bar{x}+t(x_k-\bar{x}))-\bar{H}\|\le tL_H\|x_k-\bar{x}\|$. Hence we deduce that
	\begin{equation}\label{eq:lem1a}
		\|\nabla f(x_k)\| \le \|\bar{H}(x_k-\bar{x})\| + \|\nabla^2f(\bar{x}+t(x_k-\bar{x}))-\bar{H}\| \|x_k-\bar{x}\|\le \|\bar{H}(x_k-\bar{x})\| + L_H\|x_{k}-\bar{x}\|^2.
	\end{equation}
	Therefore, by \cref{eq:localer}, we deduce that
	\begin{equation}\label{eq:upperbound_dist}
        \rho \|x_k-\bar{x}\| - L_H \|x_k-\bar{x}\|^2 \le \|\nabla f(x_k)\|  - L_H \|x_k-\bar{x}\|^2 \overset{\cref{eq:lem1a}}{\le} \|\bar{H}(x_k-\bar{x})\|.	\end{equation}        
	Since $\|x_k-\bar{x}\|$ tends to $0$, we deduce that for $k$ large enough:
	$$ \frac{\rho}{2} \|x_k-\bar{x}\| \le \|\bar{H}(x_k-\bar{x})\| \le \sqrt{\lambda_{\max}(\bar{H})} \|\sqrt {\bar H}(x_k-\bar{x})\|  .$$
\endproof

Let us now define the semi-norm:
\begin{equation}\label{eq:seminorm}
	\|x\|^2_{\bar H}:= x^\top \bar H x.
\end{equation}
Notice that by Lemma \ref{lem:1}, under \Cref{ass:local2a}, when $k$ is large enough, $\|\cdot\|_{\bar H}$ is a norm for $x_k-\bar{x}$ as we have that $\|x_k-\bar{x}\|_{\bar H}=0$ if and only if $\|x_k-\bar{x}\|=0$.

\begin{proposition}\label{prop:3a} 
{{Assume that \Cref{assumption:local1}, \Cref{assumption:s}}, \Cref{assumption:local2}$(i)$ and \Cref{ass:local2a} hold. Then for $k$ large enough:
	$$\|x_{k+1}-\bar{x}\|_{\bar H}\le \left(\sqrt{1-\frac{\lambda_0}{4\lambda_{\max}(\bar{H})\left(\sqrt{\frac{n}{s}}+\mathcal{C}\right)^2}} \right)\|x_{k}-\bar{x}\|_{\bar H}$$
	holds with probability at least  $1-6(\exp(-s)+\exp(-\frac{\mathcal{C}_0}{4}s))$.}
\end{proposition}

\proof{Proof.}
\begin{align}\label{eq:a}
	&\sqrt{\bar{H}}(x_{k+1}-\bar{x}) = 	\sqrt{\bar{H}}(x_{k+1}-x_k) +	\sqrt{\bar{H}}(x_{k}-\bar{x}) \nonumber\\
	&=	-\sqrt{\bar{H}}P_k^\top(P_kH_k P_k^\top+\eta_kI_s)^{-1}P_kg_k  +	\sqrt{\bar{H}}(x_{k}-\bar{x})  \nonumber\\
	&=  -\sqrt{\bar{H}}P_k^\top(P_kH_k P_k^\top+\eta_k I_s)^{-1}P_kH_k(x_k-\bar{x}) +  \sqrt{\bar{H}}P_k^\top(P_kH_k P_k^\top+\eta_k I_s)^{-1}P_k(g_k -H_k(x_k-\bar{x}))\\ 
	&\ \ \ +\sqrt{\bar{H}}(x_{k}-\bar{x})\nonumber\\
	&= -A+B+\sqrt{\bar{H}}(x_{k}-\bar{x}),
\end{align}
where $A:=\sqrt{\bar{H}}P_k^\top(P_kH_k P_k^\top+\eta_k I_s)^{-1}P_kH_k(x_k-\bar{x})$ and $B:=\sqrt{\bar{H}}P_k^\top(P_kH_k P_k^\top+\eta_k I_s)^{-1}P_k(g_k -H_k(x_k-\bar{x}))$. First let us bound $B$. In order to do so, we bound $\|P_k^\top(P_kH_k P_k^\top+\eta_k I_s)^{-1}P_k\|$. Notice that 
from $P_kH_k P_k^\top\succ 0$, $\eta_k\ge 0$ and  Proposition~\ref{prop:1}, we have
\begin{equation}\label{eq:aa}
  \|P_k^\top(P_kH_k P_k^\top+\eta_k I_s)^{-1}P_k\| \le  \|P_k^\top(P_kH_k P_k^\top)^{-1}P_k\|  \le
 {\frac{\|P_k^\top \|^2 }{\frac{n}{8s}\sigma\bar{\lambda}}}
\end{equation}  
with probability at least $1-6\exp(-s)$.
{
  Therefore, by Lemma \ref{lem:concentration_of_PP}, we have
\begin{equation}\label{eq:aa1}
	\|P_k^\top(P_kH_k P_k^\top+\eta_k I_s)^{-1}P_k\| \le	\frac{8\mathcal{\bar{C}} }{\sigma\bar{\lambda}}.
\end{equation}}

By Taylor expansion at $\bar{x}$ of $\nabla f$, as in \cref{eq:taylorgrad}, and by subtracting $H_k(x_k-\bar{x})$ to both sides, we obtain by \Cref{assumption:local1} that 
\begin{equation}\label{eq:aa2}
	\|g_k -H_k(x_k-\bar{x})\|\le \int_0^1 \|\nabla^2f(\bar{x}+t(x_k-\bar{x}))-\nabla^2 f(\bar{x})\|\|x_k-\bar{x}\|dt=O(\|x_k-\bar{x}\|^2).
\end{equation}
Hence, by \cref{eq:aa} and \cref{eq:aa2}, there exists a constant $\beta_1>0$ such that
\begin{equation}\label{eq:ab}
	B\le\| \sqrt{\bar{H}}\|\|P_k^\top(P_kH_k P_k^\top+\eta_k I_s)^{-1}P_k\|\|(g_k -H_k(x_k-\bar{x}))\| \le \beta_1 \|x_k-\bar{x}\|^2.
\end{equation}

Let us now bound $A=\sqrt{\bar{H}}P_k^\top(P_kH_k P_k^\top+\eta_k I_s)^{-1}P_kH_k(x_k-\bar{x})$. Let us furthermore decompose $A=A_1+A_2$ such that
\begin{align}\label{eq:b}
	 & \sqrt{\bar{H}}P_k^\top(P_kH_k P_k^\top+\eta_k I_s)^{-1}P_kH_k(x_k-\bar{x}) \nonumber\\
  & =   \sqrt{\bar{H}}P_k^\top(P_k\bar H P_k^\top+\eta_k I_s)^{-1}P_k\bar{H}(x_k-\bar{x})  +
\sqrt{\bar{H}}P_k^\top ((P_kH_k P_k^\top+\eta_k I_s)^{-1}P_k(H_k - \bar{H}) (x_k-\bar{x}) .
\end{align}
Notice that by \Cref{assumption:local1}, we have that $\|({\bar H}-{{H_k}})\|$ tends to $0$.
Therefore, we deduce from \cref{eq:aa} and \cref{eq:ab} that 
\begin{equation*}
\|\sqrt{\bar{H}}P_k^\top ((P_kH_k P_k^\top+\eta_k I_s)^{-1}P_kH_k-(P_k\bar H P_k^\top +\eta_k I_s)^{-1}P_k\bar{H}) (x_k-\bar{x})\| =o(\|x_k-\bar{x}\|).
\end{equation*}
Therefore by \cref{eq:a}, \cref{eq:ab} and \cref{eq:b}, we deduce that
$$\sqrt{\bar{H}}(x_{k+1}-\bar{x})=-A+B+\sqrt{\bar{H}}(x_{k}-\bar{x})=-A_1+\sqrt{\bar{H}}(x_{k}-\bar{x})+ o(\|x_k-\bar{x}\|). $$
Hence, by evaluating the norm of $A_2$ as $o(\|x_k-\bar{x}\|)$, we deduce that with probability at least $1-6(\exp(-s)+\exp(-\frac{\mathcal{C}_0}{4}s))$
\begin{equation*}
	\|\sqrt{\bar{H}}(x_{k+1}-\bar{x})\| \le \left\|\left(I_n -\sqrt{\bar{H}}P_k^\top(P_k\bar{H} P_k^\top+\eta_k I_s)^{-1}P_k \sqrt{\bar{H}} \right)\sqrt{\bar{H}}(x_k-\bar{x})\right\| + o(\|x_k-\bar{x}\|).
\end{equation*}
We can write
\begin{align*}
	 \left(I_n -\sqrt{\bar{H}}P_k^\top(P_k\bar{H} P_k^\top+\eta_k I_s)^{-1}P_k \sqrt{\bar H} \right)\sqrt{\bar H}&(x_k-\bar{x})= \left(I_n -\sqrt{\bar{H}}P_k^\top(P_k\bar H P_k^\top)^{-1}P_k \sqrt{\bar H} \right)\sqrt{\bar H}(x_k-\bar{x}) \\
	& -\sqrt{\bar{H}}P_k^\top ((P_k\bar H P_k^\top+\eta_k I_s)^{-1} -(P_k\bar H P_k^\top)^{-1}  )P_k {\bar H}(x_k-\bar{x}) .
\end{align*}
Hence, using the same reasoning as before, we obtain that
\begin{equation}\label{eq:c}
	\|\sqrt{\bar{H}}(x_{k+1}-\bar{x})\| \le \left\|\left(I_n -\sqrt{\bar{H}}P_k^\top(P_k\bar H P_k^\top)^{-1}P_k \sqrt{\bar H} \right)\sqrt{\bar H}(x_k-\bar{x})\right\| + o(\|x_k-\bar{x}\|).
\end{equation}
Notice that $\sqrt{\bar{H}}P_k^\top(P_k\bar H P_k^\top)^{-1}P_k \sqrt{\bar H}$ is an orthogonal projection, hence
$$\left\|\left(I_n -\sqrt{\bar{H}}P_k^\top(P_k\bar H P_k^\top)^{-1}P_k \sqrt{\bar H} \right)\sqrt{\bar H}(x_k-\bar{x})\right\|^2 = \|\sqrt{\bar H}(x_k-\bar{x})\|^2 - \|\sqrt{\bar{H}}P_k^\top(P_k\bar H P_k^\top)^{-1}P_k {\bar H}(x_k-\bar{x}) \|^2. $$
Then similarly to the proof of Proposition \ref{prop:3} and similarly to \cref{eq:align}, we have that with probability at least $1-6(\exp(-s)+\exp(-\frac{\mathcal{C}_0}{4}s))$,
$$ \|\sqrt{\bar{H}}P_k^\top(P_k\bar H P_k^\top)^{-1}P_k {\bar H}(x_k-\bar{x})\|^2 = (x_k-\bar{x})^\top \Pi (x_k-\bar{x}) ,$$
and
$$ \|\sqrt{\bar{H}}P_k^\top(P_k\bar H P_k^\top)^{-1}P_k {\bar H}(x_k-\bar{x}) \|^2 \ge \frac{\lambda_0}{2\lambda_{\max}(P_k\bar{H}P_k^\top)}\|\sqrt{\bar {H}}(x_k-\bar{x})\|^2, $$
where $\lambda_0$ is the first non-zero eigenvalue of $\bar{H}$. Therefore, we have that
$$\left\|\left(I_n -\sqrt{\bar{H}}P_k^\top(P_k\bar H P_k^\top)^{-1}P_k \sqrt{\bar H} \right)\sqrt{\bar H}(x_k-\bar{x})\right\| \le  \sqrt{1-\frac{\lambda_0}{2\lambda_{\max}(P_k\bar{H}P_k^\top)}}\|\sqrt{\bar H}(x_k-\bar{x})\|.$$
Therefore, by \cref{eq:c}, we have that 
$$\|\sqrt{\bar{H}}(x_{k+1}-\bar{x})\| \le \sqrt{1-\frac{\lambda_0}{2\lambda_{\max}(P_k\bar{H}P_k^\top)}}\|\sqrt{\bar H}(x_k-\bar{x})\|+ o(\|x_k-\bar{x}\|).  $$
By Lemma \ref{lem:1}, we have $o(\|x_k-\bar{x}\|)=o(\|\sqrt{\bar{H}}(x_{k}-\bar{x})\|)$, hence we deduce that when $k$ is large enough,
$$\|\sqrt{\bar{H}}(x_{k+1}-\bar{x})\| \le \sqrt{1-\frac{\lambda_0}{4\lambda_{\max}(P_k\bar{H}P_k^\top)}}\|\sqrt{\bar H}(x_k-\bar{x})\|.  $$
We complete the proof using \cref{eq:PHP}.
\endproof

\subsection{Impossibility of local super-linear convergence in general}\label{sec:local2}

In this section we will prove that when $f$ is strongly convex locally around the strict local minimizer $\bar{x}$, we cannot aim, with high probability, at local super-linear convergence using random subspace. More precisely, the goal of this section is to prove that there exists a constant $c>0$ such that when $k$ is large enough, we have that with probability $1-2\exp(-\frac{\mathcal{C}_0}{4})-2\exp(-s)$,
$$\|x_{k+1}-\bar{x}\| \ge c \|x_k-\bar{x}\| .$$
From that, we will easily deduce that there exists a constant $c'$ such that
$$f(x_{k+1}) -f(\bar{x}) \ge c' (f(x_k) -f(\bar{x}))$$
holds with high probability when $k$ is large enough. This will prove that the results obtained in the previous section are optimal when $f$ is locally strongly-convex.
Indeed, by local strong-convexity of $f$ and Hessian Lipschitz continuity (i.e. \Cref{assumption:local1}), there exists $l_2\ge l_1>0$ such that for $k$ large enough,
$$ l_1\|x_k-\bar{x}\|^2 \le f(x_k)-f(\bar{x}) \le l_2 \|x_k-\bar{x}\|^2.$$
This immediately proves the existence of the constant $c'$ described above. In this subsection we make the following additional assumption.

\begin{assumption}\label{assumption:local3}
	We assume that
	$$(\mathcal{C}+2)^2s<n,$$
        where $\mathcal{C}$ is the constant that appears in \cref{eq:sv_PT}. 
\end{assumption}

We recall here that for all $k$:
$$x_{k+1}=x_k-t_kP_k^\top((P_k\nabla^2 f(x_k)P_k^\top)+\eta_kI_s)^{-1}P_k \nabla f(x_k) ,$$
where $t_k$ is the step-size  and $\eta_k>0$ is a parameter that tends to $0$ when $k$ tends to infinity.

Let us fix $k$. Using a Taylor expansion of $t\mapsto\nabla f(\bar{x}+t(x_{k+1}-\bar{x}))$ around $0$, as in \cref{eq:taylorgrad}, we have that 
\begin{equation}\label{eq:0}
	\|\nabla f(x_{k+1})\| \le \int_{0}^{1} \|\nabla^2f(\bar{x}+t(x_{k+1}-\bar{x}))\| \|x_{k+1}-\bar{x}\|dt\le \int_{0}^{1} 2\lambda_{\max}(\nabla^2f(\bar x))\|x_{k+1}-\bar{x}\|dt,
\end{equation}
where $\lambda_{\max}(\cdot)$ denotes the largest eigenvalue, and the second inequality holds for $k$ large enough under \Cref{assumption:local1}.
 Hence, for $k$ large enough and under \Cref{assumption:local1},
\begin{equation}\label{eq:lowbound_dist}
  \|x_{k+1}-\bar{x}\| \ge \frac{1}{2\lambda_{\max}(\nabla^2f(\bar x))}\|\nabla f(x_{k+1})\|
\end{equation}
holds. 
Using a Taylor expansion of $\nabla f$ around $x_k$, we have that 
$$\nabla f(x_{k+1}) = \nabla f({x_{k}}) + \int_{0}^{1} \nabla^2f({x_k}+t(x_{k+1}-{x_k}))(x_{k+1}-{x_k})dt.$$
Hence, 
$$\nabla f(x_{k+1})=\nabla f(x_k)+\nabla^2f({x}_k)(x_{k+1}-x_k)+ \int_{0}^{1} (\nabla^2f({x_k}+t(x_{k+1}-{x_k}))-\nabla^2f({x}_k))(x_{k+1}-x_k) .$$
We deduce therefore that
$$\|\nabla f(x_{k+1})\| \ge \|\nabla f(x_k)+\nabla^2f({x}_k)(x_{k+1}-x_k)\| - \int_{0}^{1} \|(\nabla^2f({x_k}+t(x_{k+1}-{x_k}))-\nabla^2f({x}_k))(x_{k+1}-x_k)\| . $$
By \Cref{assumption:local1}, the Hessian is $L_H$-Lipschitz in $B_H$. Since $x_k$ and ${x}_k+t(x_{k+1}-x_k) \in B_H$ for $k$ large enough, we have  that
for $t \leq 1$,
$$\|(\nabla^2f({x_k}+t(x_{k+1}-{x_k}))-\nabla^2f({x}_k))(x_{k+1}-x_k))\|\le L_H  \|x_{k+1}-x_k\|^2.$$
Hence \cref{eq:lowbound_dist} leads to  
\begin{equation}\label{eq:1}
	\|x_{k+1}-\bar{x}\| \ge \frac{1}{2\lambda_{\max}(\nabla^2f(\bar x))} \left(  \|g_k+H_k(x_{k+1}-x_k)\| -  L_H \|x_{k+1}-x_k\|^2  \right).
\end{equation}

\begin{proposition}\label{prop:4}
	Assume that \Cref{assumption:local1} and \Cref{assumption:local3} hold and that $f$ is strongly convex locally around $\bar{x}$. There exists a constant $\beta>0$ such that if $k$ is large enough, then with probability at least $1-2\exp(-\frac{\mathcal{C}_0}{4}s)-2\exp(-s)$, we have
	$$\|g_k+H_k(x_{k+1}-x_k)\| \ge \beta \|x_{k+1}-x_k\|. $$
\end{proposition}
\proof{Proof.}
	Recalling the updated rule $x_{k+1}=x_k -  t_k P_k^\T M_k^{-1} P_k g_k$ in \Cref{alg:RSN}, we have
	$$\|g_k+H_k(x_{k+1}-x_k)\|=\|(I_n-t_kH_kP_k^\top M_k^{-1}P_k)g_k\|,$$
	where $M_k$ is defined in \cref{def:Mk}. If $k$ is large enough, $H_k$ is invertible by strong convexity of $f$. Notice that $ \|(I_n-t_kH_kP_k^\top M_k^{-1}P_k)g_k\|=\|H_k(H_k^{-1}-t_kP_k^\top M_k^{-1}P_k)g_k\|$. Hence since for any invertible matrix $A$ we have $\|Ax\|\ge \frac{\|x\|}{\|A^{-1}\|}$, we deduce that
	$$\|(I_n-t_kH_kP_k^\top M_k^{-1}P_k)g_k\| \ge \frac{1}{\|H_k^{-1}\|}\| (H_k^{-1}-t_kP_k^\top M_k^{-1}P_k)g_k \|.$$
	Furthermore, we have 
	\begin{align}\label{eq:prop4a}
		&\|(H_{k}^{-1}-t_kP_k^\top M_k^{-1}P_k)g_k\|^2\\
		&=\|H_k^{-1}g_k\|^2+\|t_kP_k^\top M_k^{-1}P_kg_k\|^2-2\langle H_k^{-1}g_k, t_kP_k^\top M_k^{-1}P_k g_k \rangle. \nonumber
	\end{align}
	Let $H_k^{-1}g_k=P_k^\top z_1+z_2$ be the orthogonal decomposition of $H_k^{-1}g_k$ on $\Img(P_k^\top)$ parallel to $\Ker(P_k)$. Since $P_kz_2=0$, we have 
	$$\langle H^{-1}_kg_k, t_kP_k^\top M_k^{-1}P_k g_k \rangle=\langle P_k^\top z_1, t_kP_k^\top M_k^{-1}P_k g_k \rangle .$$
	Hence, by \cref{eq:prop4a}, we deduce that
	\begin{align}\label{eq:prop4b}
		&\|(H_{k}^{-1}-t_kP_k^\top M_k^{-1}P_k)g_k\|^2 \\
		&\ge \|H_k^{-1}g_k\|^2+\|t_kP_k^\top M_k^{-1}P_kg_k\|^2 -2 \|P_k^\top z_1\|\|t_k P_k^\top M_k^{-1}P_k g_k\|. \nonumber
	\end{align}
	Since $H^{-1}_kg_k=P_k^\top z_1+z_2$ with $P_kz_2=0$, we have that $P_kH^{-1}_kg_k=P_kP_k^\top z_1$. Which implies (since $P_kP_k^\top $ is invertible with probability $1$) that $z_1=(P_kP_k^\top)^{-1}P_kH^{-1}_kg_k$. Hence
	$$\|P_k^\top z_1\|=\|P_k^\top(P_kP_k^\top)^{-1}P_kH^{-1}_kg_k\|\le \|P_k^\top(P_kP_k^\top)^{-1}\| \|P_kH^{-1}_kg_k\|. $$
	By Lemma \ref{lem:JLL_random_matrix}, we have that with probability at least $1-2\exp(-\frac{\mathcal{C}_0}{4}s)$ that
	$ \|P_kH^{-1}_kg_k\|\le 2\|H^{-1}_kg_k\|$. Furthermore, by writing the singular value decomposition, $U\Sigma V^\top$, of $P_k^\top$, we have that ~$\|P_k^\top(P_kP_k^\top)^{-1}\|=\|U\Sigma^{-1}V^\top\| =\frac{1}{\sigma_{\min}(P_k^\top)}$. Since  
	 $\sigma_{\min}(P_k^\top)\ge \sqrt{\frac{n}{s}}-\mathcal{C}$ holds with probability at least $1-2e^{-s}$ (we only consider the first equation of \cref{eq:sv_PT}), we deduce that 
	$$\|P_k^\top z_1\| \le \frac{2}{\sqrt{\frac{n}{s}}-\mathcal{C}}\|H_k^{-1}g_k\|.$$
	Hence, from \cref{eq:prop4b} we have
	\begin{align}\label{eq:prop4c}
		&\|(H_k^{-1}-t_kP_k^\top M_k^{-1}P_k)g_k\|^2 \\
		&\ge \|H_k^{-1}g_k\|^2+\|t_kP_k^\top M_k^{-1}P_kg_k\|^2 - \frac{4}{\sqrt{\frac{n}{s}}-\mathcal{C}}\|H_k^{-1}g_k\|\|t_kP_k^\top M_k^{-1}P_k g_k\| \nonumber \\
		& \ge \left(1-\frac{2}{\sqrt{\frac{n}{s}}-\mathcal{C}}\right)\|H_k^{-1}g_k\|^2 + \left(1-\frac{2}{\sqrt{\frac{n}{s}}-\mathcal{C}}\right) \|t_kP_k^\top M_k^{-1}P_kg_k\|^2, \nonumber
	\end{align}
	where we used that $2ab \le a^2+b^2$ in the last inequality, and that $\left(1-\frac{2}{\sqrt{\frac{n}{s}}-\mathcal{C}}\right)>0$ holds by \Cref{assumption:local3}. Hence, from \eqref{eq:prop4c} we proved that
	\begin{align*}
		&\|(I_n-t_kH_kP_k^\top M_k^{-1}P_k)g_k\|^2\ge \frac{1}{\|H_k^{-1}\|^2}\left(1-\frac{2}{\sqrt{\frac{n}{s}}-\mathcal{C}}\right) \|t_kP_k^\top M_k^{-1}P_kg_k\|^2 \\
		&=  \frac{1}{\|H_k^{-1}\|^2}\left(1-\frac{2}{\sqrt{\frac{n}{s}}-\mathcal{C}}\right)\|x_{k+1}-x_k\|^2. 
	\end{align*}
	That is
	\begin{equation}\label{eq:2}
		\|g_k+H_k(x_{k+1}-x_k)\| \ge  \frac{\sqrt{1-\frac{2}{\sqrt{\frac{n}{s}}-\mathcal{C}}}}{\|H_k^{-1}\|}\|x_{k+1}-x_k\|.
	\end{equation}
	Considering $k$ large enough, as $x_k$ tends to $\bar{x}$, we can bound, using \Cref{assumption:local1}, $\frac{1}{\|H_k^{-1}\|}\ge \frac{1}{2\|{\bar{H}}^{-1}\|}$, where we recall that $\bar{H}=\nabla^2 f(\bar{x})$, which ends the proof.
\endproof

\begin{theorem}\label{thm:3}
	Assume that \Cref{assumption:local1} and \Cref{assumption:local3} hold and that $f$ is locally strongly convex around $\bar{x}$. There exists a constant $c>0$ such that for $k$ large enough, 
	$$\|x_{k+1}-\bar{x}\| \ge c \|x_k-\bar{x}\| $$
	holds with probability at least $1-2\exp(-\frac{\mathcal{C}_0}{4}s)-2\exp(-s)$.
\end{theorem}
\proof{Proof.}
From \cref{eq:1} and Proposition \ref{prop:4} we deduce that with probability at least $1-2\exp(-\frac{\mathcal{C}_0}{4}s)-2\exp(-s)$, when $k$ is large enough
$$ \|x_{k+1}-\bar{x}\| \ge \frac{1}{2\lambda_{\max}(\nabla^2f(\bar x))} \left( \beta -  L_H \|x_{k+1}-x_k\|  \right)\|x_{k+1}-x_k\|.$$
Since $\beta >0$, we have that for $k$ large enough so as to yield $L_H \|x_{k+1}-x_k\| \leq \beta/2$,
$$ \|x_{k+1}-\bar{x}\| \ge \frac{1}{2\lambda_{\max}(\nabla^2f(\bar x))}  \frac{\beta}{2} \|x_{k+1}-x_k\|.$$
Hence
\begin{equation}\label{eq:lowbound_dist1}
  \|x_{k+1}-\bar{x}\| \ge \frac{\beta}{4\lambda_{\max}(\bar{H})} \|x_{k+1}-x_k\|.
\end{equation} 
Since $f$ is assumed to be strongly convex, 
for all $\alpha \in (0,1)$, as $g_k^\T d_k\le 0$. Hence we have that $t_k=1$Now we notice that 
\begin{equation}\label{eq:thm3a}
	\|x_{k+1}-x_k\|=t_k\|P_k^\top M_k^{-1}P_kg_k\|\ge t_k \sigma_{\min}(P_k^\top)\|M_k^{-1}\|\|P_k g_k\|.
\end{equation}
Using Lemma \ref{lem:JLL_random_matrix} (with $\varepsilon=1/2$) and the bound  \cref{eq:sv_PT} on $\sigma_{\min}(P_k^\top)$, we have that
\begin{equation}\label{eq:thm3b}
	t_k\sigma_{\min}(P_k^\top)\|M_k^{-1}\|\|P_k g_k\| \ge t_k\left(\sqrt{\frac{n}{s}}-\mathcal{C}\right) \|M_k^{-1}\| \frac{1}{2}\|g_k\|.
\end{equation}
Since $x_k$ converges to $\bar{x}$ and the Hessian is Lipschitz continuous, we have that $H_k$ converges to $\bar{H}$. Therefore, when $k$ is large enough, we have $\|M_k^{-1}\| \ge \frac{1}{2}\|(P_k\bar{H}P_k^\top)^{-1}\|=\frac{1}{2}\|{\bar{M}}^{-1}\|$,
where $\bar{M}:=P_k\bar{H}P_k^\top$.
Since 
$$0 \prec \bar{M} \preceq \lambda_{\max}(\bar{H})P_kP_k^\top ,$$
we deduce by Lemma \ref{lem:concentration_of_PP}
\begin{equation}\label{eq:thm3c}
	\|M_k^{-1}\| \ge \frac{1}{2\mathcal{C}\lambda_{\max}(\bar{H})\frac{n}{s}}.
\end{equation}
Hence, by \cref{eq:lowbound_dist1,eq:thm3a,eq:thm3b,eq:thm3c} we have that there exists a constant $\kappa_2>0$ such that 
$$\|x_{k+1}-\bar{x}\| \ge \kappa_2 \|g_k\| .$$
By \cref{eq:taylorgrad} we have that
$$g_k = \bar{H}(x_k-\bar{x}) + \int_0^1(\nabla^2f(\bar{x}+t(x_k-\bar{x}))-\bar{H})(x_k-\bar{x})dt.$$
Hence, since $f$ is assumed to be locally strongly convex, by  \Cref{assumption:local1} we have that for $k$ large enough:
$$\|g_k\|\ge \frac{\lambda_{\min}(\bar{H})}{2}\|x_k-\bar{x}\|. $$
Using \eqref{eq:localt_k}, we have
$$ f(x_k) - f(x_k+t_k'd_k) + \alpha t_k' g_k^\T d_k
\ge  \frac{\bar{\mathcal{C}}n}{2s} L_H {t_k'}^2\norm{d_k} \left(\frac{c_2s\norm{g_k}^\gamma}{\bar{\mathcal{C}}L_H n\norm{d_k}} - t_k'  \right) \norm{M_k^{-1} P_kg_k}^2,$$
and since $f$ is assume to be strongly convex,  $\frac{\|g_k\|^{\gamma}}{\|d_k\|}$ is in the order of $\mathcal{O}(\frac{1}{\|g_k\|^{1-\gamma}})$, hence $t_k$ is bounded below by some constant for $k$ large enough.
Hence we have for $k$ large enough that
$$\|x_{k+1}-\bar{x}\| \ge \frac{1}{2}\kappa_2 \lambda_{\min} (\bar{H})\|x_k-\bar{x}\|,$$
which concludes the proof.
\endproof

We have the following deterministic corollary:
\begin{corollary}\label{cor:3}
	Assume that \Cref{assumption:local1} and \Cref{assumption:local3} hold and that $f$ is locally strongly convex around $\bar{x}$. Then for $k$ large enough, 
	$$\mathbb{E}(\|x_{k+1}-\bar{x}\| ) \ge \bar{c}\mathbb{E}(\|x_k-\bar{x}\|) ,$$
	where $\bar{c}=(1-2\exp(-\frac{\mathcal{C}_0}{4}s)-2\exp(-s))c$ ($c$ is the same constant as in Theorem \ref{thm:3}), and where the expectation is taken with respect to the random variables $P_0,\cdots,P_k$.
\end{corollary}
\proof{Proof.}
	The proof is very similar to the proof of Corollary \ref{cor:2}. Let us consider the random variable $\mathbb{E}[\|x_{k+1}-\bar{x}\|\ |\ P_0,\cdots,P_{k-1}]$. Let $\mathcal{E}=\{\|x_{k+1}-\bar{x}\| \ge \bar{c}\|x_k-\bar{x}\|\ |\ x_k\}$ be an event with respect to the random variable $P_k$. Using the fact that $\|x_{k+1}-\bar{x}\|\ge0,$ we obtain that
	\begin{align*}
		\mathbb{E}\left[\|x_{k+1}-\bar{x}\|\ |\ P_0,\cdots,P_{k-1}\right] &= \mathbb{E}\left[\|x_{k+1}-\bar{x}\|\ |\ P_0,\cdots,P_{k-1}, {\mathcal{E}} \right] P(\mathcal{E})\\
		& \ \ \ +\mathbb{E}\left[\|x_{k+1}-\bar{x}\|\ |\ P_0,\cdots,P_{k-1},{\bar{\mathcal{E}}} \right] (1-P(\mathcal{E})) \\
		&\ge \bar{c}\|x_{k}-\bar{x}\|
	\end{align*}
	Taking the expectation with respect to $P_0,\cdots,P_{k-1}$ leads to the result.
\endproof

{\subsection{The rank deficient case}\label{sec:lowrank}
Previously we proved that when $f$ is locally strongly convex, super-linear convergence cannot hold for RS-RNM. Here we prove that when the Hessian $\bar{H}$ at the local optimum $\bar{x}$ is rank deficient, then RS-RNM can achieve super-linear convergence. 
In this whole subsection, we assume that \Cref{assumption:local1} and \Cref{ass:local2a} are satisfied.
We also denote by $r$ ($<n$) the rank of $\bar{H}$.
{
Notice that, as a special case of $r<n$, one can consider
``functions with low dimensionality''\footnote{They are also called objectives with ``active subspaces" \cite{constantine2014active}, or ``multi-ridge" \cite{fornasier2012learning}. } \cite{wang2016bayesian}. For such functions, there exists a projection matrix $\Pi \in \mathbb{R}^{n\times n}$ with $\mathrm{rank}(\Pi) < n$ such that
\begin{equation} \label{eq:lowdim} \textstyle
	 \forall x \in \mathbb{R}^n, ~f(x) = f(\Pi x).
\end{equation}
Such functions are frequently encountered in many applications. For example, 
the loss functions of neural networks often have low rank Hessians \cite{gur2018gradient,sagun2017empirical,papyan2018full}. This phenomenon is also prevalent in other areas such as hyper-parameter optimization for neural networks \cite{bergstra2012random}, heuristic algorithms for combinatorial optimization problems \cite{hutter2014efficient}, complex engineering and physical simulation problems as in climate modeling \cite{knight2007association}, and policy search \cite{frohlich2019bayesian}.}
}

We first prove the following lemma which is very similar to \Cref{lem:1}.
\begin{lemma}\label{lem:1a}
{	We have, under \Cref{assumption:local1} and \Cref{ass:local2a}, that for $k$ large enough:
	$$ \frac{\rho}{2} \|x_k-\bar{x}\| \le \|{\bar{H}}(x_k-\bar{x})\| .$$
	Furthermore,
	$$\|g_k\| \le 2\lambda_{\max}(\bar{H})\|x_k-\bar{x}\|. $$}
\end{lemma}
\proof{Proof.}
As in the proof of Lemma~\ref{lem:1}, we have \eqref{eq:upperbound_dist}, i.e.,
$$\rho \|x_k-\bar{x}\| - L_H \|x_k-\bar{x}\|^2 
{\le} \|\bar{H}(x_k-\bar{x})\|.$$
	Since $\|x_k-\bar{x}\|$ tends to $0$, we deduce that for $k$ large enough:
	$$ \frac{\rho}{2} \|x_k-\bar{x}\| \le \|\bar{H}(x_k-\bar{x})\|.$$
	The other inequality is  easy to deduce from \cref{eq:taylorgrad}, as in \cref{eq:0}:
	\begin{equation}\label{eq:upper_gnorm}
	\|g_k\| \le  \|\bar{H}\|\|x_k-\bar{x}\| + L_H \|x_k-\bar{x}\|^2 \le 2\lambda_{\max}(\bar{H}) \|x_k-\bar{x}\| ,
	\end{equation}        
	 when $k$ is large enough such that $L_H \|x_k-\bar{x}\| \le \lambda_{\max}(\bar{H})$ holds. 
\endproof
The next lemma is the key to prove super-linear convergence. Notice that since $s\ge r$, we have that with probability one $\sigma_{\min}(P_k^1)>0$. 
\begin{lemma}\label{lem:2}
{	Under \Cref{assumption:local1} and \Cref{ass:local2a}. If $s\ge r$,
	we have that for $k$ large enough, with probability at least $1-2\exp(-s)$:
	$$\|P_k g_{k+1}\| \ge  \rho\frac{\sigma_{\min}(P_k^1)}{8\lambda_{\max}(\bar{H})}\|g_{k+1}\|, $$
	where $P_k^1\in \mathbb{R}^{s \times r}$ is an $s \times r$ i.i.d. Gaussian matrix having the same distribution with $P_k$.}
\end{lemma}
\proof{Proof.}
	By \cref{eq:taylorgrad} applied at $k+1$, we have that
	$$ \nabla f(x_{k+1}) =\int_{0}^{1} \nabla^2f(\bar{x}+t(x_{k+1}-\bar{x}))(x_{k+1}-\bar{x})dt.$$
	Hence,
	$$P_k g_{k+1} = P_k\bar{H}(x_{k+1}-\bar{x}) + \int_{0}^{1} P_k(\nabla^2f(\bar{x}+t(x_{k+1}-\bar{x}))-\bar{H})(x_{k+1}-\bar{x}),$$
which leads to
	\begin{equation}\label{eq:c1}
		\|P_k g_{k+1}\| \ge \|P_k\bar{H}(x_{k+1}-\bar{x})\| - L_H\|P_k\|\| x_{k+1}-\bar{x}\|^2.
	\end{equation}
	Let $UDU^\top =\bar{H}$ be the diagonal decomposition of $\bar{H}$. Since $\bar{x}$ is a strict local minimizer, by \Cref{ass:local2a}, for $k$ large enough, $U$ is an orthogonal matrix independent of $P_k$, and hence, $\tilde{P}_k:=P_kU$ is an i.i.d. random Gaussian matrix with the same distribution as $P_k$. Let $y_{k+1}=U^\top (x_{k+1}-\bar{x})$. We have that 
	\begin{equation}\label{eq:aa1b}
		\bar{H}(x_{k+1}-\bar{x})=UDy_{k+1}\quad \mbox{ and thus, } 	\quad P_k\bar{H}(x_{k+1}-\bar{x})=\tilde{P}_kDy_{k+1}.
	\end{equation}
	Furthermore, since $D$ has rank $r<n$, we can write $Dy_{k+1}=\begin{pmatrix}
		z_{k+1}\\
		0
	\end{pmatrix}$, where $z_{k+1}\in \mathbb{R}^r$.
	We have therefore that
	\begin{equation}\label{eq:d}
		\|P_k\bar{H}(x_{k+1}-\bar{x})\|=\|{P}_k^1 z_{k+1}\|,
	\end{equation}
	where ${P}_k^1 \in \mathbb{R}^{s\times r}$ is a submatrix of $\tilde{P}_k$, i.e., $\tilde{P}_k=\begin{pmatrix}
		{P}_k^1 & P_k^2
	\end{pmatrix}$.
	Notice that from the definition of $y_{k+1}$ and $z_{k+1}$, we have, by orthogonality of $U$, that
	$$\|z_{k+1}\|=\|Dy_{k+1}\|\overset{\cref{eq:aa1}}{=}\|\bar{H}(x_{k+1}-\bar{x} )\| \ge  \frac{\rho}{2} \|x_{k+1}-\bar{x}\| ,$$
	where the inequality follows from Lemma \ref{lem:1a}.
	Hence, from \cref{eq:c1} and \cref{eq:d}, we deduce that
	$$ \|P_k g_{k+1}\| \ge \rho\frac{\sigma_{\min}(P_k^1)}{2}\|x_{k+1}-\bar{x}\| -L_H \|P_{k}\|\| x_{k+1}-\bar{x}\|^2 .$$
	Using that $\|P_k\|$ is bounded, with probability at least $1-2\exp(-s)$, by Lemma \ref{lem:concentration_of_PP}, we deduce, as in the proof of Lemma \ref{lem:1a}, that for $k$ large enough:
	$$ \|P_k g_{k+1}\| \ge \rho\frac{\sigma_{\min}(P_k^1)}{4}\|x_{k+1}-\bar{x}\| \overset{\eqref{eq:upper_gnorm}}{\ge} \rho\frac{\sigma_{\min}(P_k^1)}{4} \frac{\|g_{k+1}\|}{2\lambda_{\max}(\bar{H})} .$$
That is:
	$$\|P_k g_{k+1}\| \ge \rho\frac{\sigma_{\min}(P_k^1)}{8\lambda_{\max}(\bar{H})}\|g_{k+1}\|. $$
\endproof

Similarly, we have the following lemma.
\begin{lemma}\label{lem:3}
{	Let $M \in \mathbb{R}^{n \times n}$ be any matrix. Under \Cref{assumption:local1} and \Cref{ass:local2a}, if $k$ is large enough and $s\ge r$, we have
	$$\frac{ \sigma_{\min}({P}_k^1)}{2}\|H_kM\| \le \|P_kH_kM\| .$$}
\end{lemma}
\proof{Proof.}
	The proof is very similar to the proof of Lemma \ref{lem:2}. We have
	\begin{equation}\label{eq:ab1}
		\|P_kH_kM\|  \ge \|P_k \bar{H}M\| -  \|P_k (H_k-\bar{H})M\|.
	\end{equation}
	Let $UDU^\top =\bar{H}$ be the diagonal decomposition of $\bar{H}$.
        Similarly to the proof of Lemma \ref{lem:2},
        $\tilde{P}_k:=P_kU$ is an i.i.d. random Gaussian matrix with the same distribution as $P_k$. Using $N:=U^\top M$, we have that $ P_k\bar{H}M=\tilde{P}_kDN$. Furthermore, since $D$ has rank $r<n$, we can write $DN=\begin{pmatrix}
		\tilde{N}\\
		0
	\end{pmatrix}$, where $\tilde{N}\in \mathbb{R}^{r\times n}$.
	We have therefore that
	\begin{equation}\label{eq:d1}
		\|P_k\bar{H}M\|=\|{P}_k^1 \tilde{N}\|,
	\end{equation}
	where ${P}_k^1 \in \mathbb{R}^{s\times r}$ is a submatrix of $\tilde{P}_k$, i.e., $\tilde{P}_k=\begin{pmatrix}
		{P}_k^1 & P_k^2
	\end{pmatrix}$.
	Therefore 
	\begin{equation}\label{eq:ac}
		\|P_k\bar{H}M\| \ge \sigma_{\min}({P}_k^1)\|\tilde{N}\|= \sigma_{\min}({P}_k^1) \|DN\|= \sigma_{\min}({P}_k^1) \|\bar{H}M\|,
	\end{equation}
where the last equality holds by orthogonality of $U$. We deduce therefore, from \cref{eq:ab1} and \cref{eq:ac} that
	$$ \|P_kH_kM\| \ge  \sigma_{\min}({P}_k^1) \|{H}_kM\| -  \sigma_{\min}({P}_k^1) \|(\bar{H}-H_k)M\| -\|P_k (H_k-\bar{H})M\|.$$
Since $H_k$ tends to $\bar{H}$, we have the desired result for $k$ large enough. 
\endproof

The next lemma, similar to Lemma 5.2 of \cite{ueda2010convergence}, is needed to control $\eta_k=c_1\Lambda_k  + c_2\norm{g_k}^\gamma$, where $\Lambda_k=\max(0,-\lambda_{\min}(P_kH_kP_k^\top))$.
\begin{lemma}\label{lem:Lambda_bound}
{	Under \Cref{assumption:local1},
	for $k$ large enough, we have that with probability at least $1-2\exp(-s)$,
	\begin{equation*}
		\Lambda_k \le \frac{\bar{\mathcal{C}}n}{s}  L_H \|x_k-\bar{x}\|.
	\end{equation*}}
\end{lemma}
\proof{Proof.}
	The result is obvious when $\Lambda_k=0$. Let us consider the case $\Lambda_k>0$.
	Let $\lambda_k = (\lambda_k^{(1)},\dots, \lambda_k^{(s)})$ be a vector of eigenvalues of $P_k \bar{H} P_k^\top$ and we write the eigenvalue decomposition of $P_k\bar{H} P_k^\top$ as follows:
	\begin{equation*}
		P_k \bar{H} P_k^\top = U_k^\top diag(\lambda_k) U_k.
	\end{equation*}
	Notice that $\lambda_{\min}(P_k H_k P_k^\top)I_s - U_k P_k H_k P_k^\top U_k^\top$
	is singular. Furthermore,
	$$\lambda_{\min}(P_k H_k P_k^\top)I_s - diag(\lambda_k)$$
	is not singular as $\lambda_{\min}(P_k H_k P_k^\top)<0$ by assumption and $diag(\lambda_k)$ is positive. We define
	\begin{equation*}
		A_k = (\lambda_{\min}(P_k H_k P_k^\top)I_s - diag(\lambda_k))^{-1} (\lambda_{\min}(P_k H_k P_k^\top)I_s - U_k P_k H_k P_k^\top U_k^\top),
	\end{equation*}
	which is therefore singular. 
	Notice furthermore that since $\lambda_{\min}(P_k H_k P_k^\top) <0$,
	\begin{equation}\label{eq:ae}
		\|(\lambda_{\min}(P_k H_k P_k^\top)I_s - diag(\lambda_k))^{-1}\| \le \frac{1}{-\lambda_{\min}(P_k H_k P_k^\top)} =\frac{1}{\Lambda_k}.
	\end{equation}
	Hence we have
	\begin{align*}
		1 &\le \|{I_s-A_k}\|\\
		&=\|I_s-(\lambda_{\min}(P_k H_k P_k^\top)I_s - diag(\lambda_k))^{-1} (\lambda_{\min}(P_k H_k P_k^\top)I_s - U_k P_k H_k P_k^\top U_k^\top)\| \\
		&=\| I_s - (\lambda_{\min}(P_k H_k P_k^\top)I_s - diag(\lambda_k))^{-1}\cdot\\
		& \hspace{10em} (\lambda_{\min}(P_k H_k P_k^\top)I_s - diag(\lambda_k) - U_k P_k (H_k-\bar{H}) P_k^\top U_k^\top)\|\\
		&=\|(\lambda_{\min}(P_k H_k P_k^\top)I_s - diag(\lambda_k))^{-1}U_k P_k (H_k-\bar{H}) P_k^\top U_k^\top\|\\
		&\overset{\cref{eq:ae}}{\le} \frac{1}{\Lambda_k} \|P_k P_k^\top\| \|H_k-\bar{H}\|\\
		&\overset{\mbox{Lemma }\ref{lem:concentration_of_PP}}{\le} \frac{1}{\Lambda_k} \frac{\bar{\mathcal{C}}n}{s}  L_H \|x_k-\bar{x}\|,
	\end{align*}
	where the first inequality is a well known inequality for a singular matrix and is proved in \cite[Lemma 5.1]{ueda2010convergence}.
\endproof

Let us recall that
$$d_k= -P_k^\top(P_k H_kP_k^\top+\eta_kI_s)^{-1}P_k  g_k, $$
and 
$$M_k=P_k H_kP_k^\top+\eta_kI_s.$$
\begin{lemma}\label{lem:dk}
{	Under \Cref{assumption:local1} and \Cref{ass:local2a}, if $s\ge r$,
	we have that for $k$ large enough, we have that with probability at least $1-2\exp(-s)$,
	$$ \|d_k\| \le \frac{4}{\sigma_{\min}(P_1^k)}\left(2+\frac{1}{c_1-1}\right)\sqrt{\frac{\bar{\mathcal{C}}n}{s}} \|x_k-\bar{x}\|,$$
where $P_k^1\in \mathbb{R}^{s \times r}$ is an $s \times r$ i.i.d. Gaussian matrix having the same distribution with $P_k$.}
\end{lemma}
\proof{Proof.}
	Notice first that by Taylor expansion of $t\mapsto \nabla f(\bar{x}+t(x_{k}-\bar{x}))$ and by \Cref{assumption:local1}, we have that
	\begin{equation}\label{eq:ad}
		\|g_{k}-\nabla f(\bar{x})-H_k(x_{k}-\bar{x})\|\le\frac{ L_H }{2}\|x_{k}-\bar{x}\|^2.
	\end{equation}
        The definition of $d_k$ leads to
	\begin{align}\label{eq:lemdk1}
		\|{d_k}\| 
		&= \|P_k^\top M_k^{-1}P_k g_k\|\nonumber\\
		&\overset{\nabla f(\bar x)=0}{=} \|{P_k^\top M_k^{-1}P_k (g_k -\nabla f(\bar x) - H_k(x_k-\bar x) + H_k(x_k-\bar x))}\|\nonumber\\
		&\le  \|{P_k}\|^2 \|{M_k^{-1}}\| \|{g_k -\nabla f(\bar x) - H_k(x_k-\bar x)}\| + \|{P_k^\top M_k^{-1}P_k H_k}\|\|{x_k-\bar x}\|\nonumber\\
		&\overset{\cref{eq:ad}}{\le} \frac{ L_H}{2} \|{P_k}\|^2 \|{M_k^{-1}}\|\|{x_k-\bar x}\|^2 + \|{P_k^\top M_k^{-1}P_k H_k}\|\|{x_k-\bar x}\|.
	\end{align}
	Let us first bound the first term in the right-hand side of \cref{eq:lemdk1}. When $k$ is large enough, with probability at least $1-2\exp(-s)$, we have by Lemma \ref{lem:concentration_of_PP}
	\begin{align*}
		\frac{ L_H}{2} \|{P_k}\|^2 \|{M_k^{-1}}\| 
		&\le \frac{ L_H}{2} \cdot \frac{\bar{\mathcal{C}}n}{s} \cdot \frac{1}{\lambda_{\min}(P_k H_k P_k^\top + c_1\Lambda_k I_s + c_2\|{g_k}\|^\gamma I_s)}\\
		&\le \frac{ L_H \bar{\mathcal{C}}n}{2c_2 s \|{g_k}\|^\gamma} \\
		&\overset{\cref{eq:localer}}{\le} \frac{ L_H\bar{\mathcal{C}} n}{2c_2 s \rho^\gamma \|x_k-\bar{x}\|^\gamma}.
	\end{align*}  
	Hence 
	\begin{equation}\label{eq:lemdk2}
		\frac{ L_H}{2} \|{P_k}\|^2 \|{M_k^{-1}}\| \|x_k-\bar{x}\|^2 \le \frac{ L_H \bar{\mathcal{C}} n}{2c_2 s \rho^\gamma } \|x_k-\bar{x}\|^{2-\gamma} .
	\end{equation}
	Next, we consider the second term $\|{P_k^\top M_k^{-1}P_k H_k}\| \|x_k-\bar{x}\|$.
	Notice that
	$$\|{P_k^\top M_k^{-1}P_k H_k}\| =\|{H_kP_k^\top M_k^{-1}P_k}\| \le \frac{2}{\sigma_{\min}(P_1^k)}\|P_kH_kP_k^\top M_k^{-1}\| \|P_k\|, $$
	where the inequality follows from Lemma \ref{lem:3}. We have
	\begin{align*}
		\|P_kH_kP_k^\top M_k^{-1}\| & = \|P_kH_kP_k^\top (P_kH_kP_k^\top +\eta_k I_s)^{-1}\| \\
		& \le  \|(P_kH_kP_k^\top+ \eta_kI_s)^\top (P_kH_kP_k^\top +\eta_k I_s)^{-1}\| + \eta_k \|(P_kH_kP_k^\top +\eta_k I_s)^{-1}\| \\
		&\le 1 + \frac{\eta_k}{\lambda_{\min}(P_k H_k P_k^\top +\eta_kI_s)} \\
		&\le 1+ \frac{ c_1\Lambda_k +c_2 \|g_k\|^{\gamma}}{ (c_1-1)\Lambda_k +c_2 \|g_k\|^{\gamma}} \\
		& \le 2+\frac{1}{c_1-1}.
	\end{align*}
	Therefore,
	$$\|{P_k^\top M_k^{-1}P_k H_k}\| \|x_k-\bar{x}\| \le \frac{2}{\sigma_{\min}(P_1^k)} \left(2+\frac{1}{c_1-1}\right)\|P_k\| \|x_k-\bar{x}\| \le  \frac{2}{\sigma_{\min}(P_1^k)}  \left(2+\frac{1}{c_1-1}\right)\sqrt{\frac{\bar{\mathcal{C}}n}{s}} \|x_k-\bar{x}\|,$$
	where the second inequality follows from Lemma \ref{lem:concentration_of_PP}. The results follows from \cref{eq:lemdk1} and \cref{eq:lemdk2} noticing that $\frac{\|x_k-\bar{x}\|^{2-\gamma}}{\|x_k-\bar{x}\|}$ tends to $0$, as $\gamma<1$, hence for $k$ large enough
	$$ \frac{ L_H\bar{\mathcal{C}} n}{2c_2 s \rho^\gamma } \|x_k-\bar{x}\|^{2-\gamma} \le \frac{2}{\sigma_{\min}(P_1^k)}\left(2+\frac{1}{c_1-1}\right)\sqrt{\frac{\bar{\mathcal{C}}n}{s}} \|x_k-\bar{x}\|. $$
\endproof

\begin{theorem}\label{th:3}
{	Under \Cref{assumption:local1} and \Cref{ass:local2a}, 
	for $k$ large enough and for any $s\ge r$, we have that with probability at least $1-2\exp(-s)$
	$$ \|x_{k+1}-\bar{x}\| \le \frac{c_2\Gamma}{\sigma^2_{\min}(P_k^1)} \|x_k -\bar{x}\|^{1+\gamma} ,$$
	where $\Gamma$ is some constant depending on $n$ and $s$, and where $P_k^1\in \mathbb{R}^{s \times r}$ is an $s \times r$ i.i.d. Gaussian matrix having the same distribution with $P_k$.}
\end{theorem}
\proof{Proof.}
	We have 
	\begin{align}\label{eq:super1}
		\|x_{k+1}-\bar{x} \| &\overset{\cref{eq:localer}}{\le} \frac{1}{\rho}\|g_{k+1}\| \notag\\
		& \le \frac{8 \lambda_{\max}(\bar{H})}{\rho^2 \sigma_{\min}(P^1_k)} \|P_k g_{k+1}\| \notag \\
		& \le  \frac{8 \lambda_{\max}(\bar{H})}{\rho^2 \sigma_{\min}(P^1_k)} \left( \|P_k (g_{k+1}-g_k-H_k(x_{k+1}-x_k))\| + \|P_kg_k +P_kH_k(x_{k+1}-x_k)\|\right),
	\end{align}
	where the first inequality holds by \cref{eq:localer}, and the second holds by Lemma \ref{lem:2}.
	
	By Lemma \ref{lem:dk} and an equation similar to \cref{eq:ad} (where $x_k$ is replaced by $x_{k+1}$ and $\bar{x}$ is replaced by $x_k$), we have that 
	\begin{equation}\label{eq:super2}
		\|P_k (g_{k+1}-g_k-H_k(x_{k+1}-x_k))\| \le  L_H  \|P_k\|\left(\frac{4}{\sigma_{\min}(P_1^k)}\left(2+\frac{1}{c_1-1}\right) \sqrt{\frac{\bar{\mathcal{C}}n}{s}}\right)^2 \|x_k-\bar{x}\|^2.
	\end{equation}
        From the updated rule $x_{k+1}=x_k -  t_k P_k^\T M_k^{-1} P_k g_k$ in \Cref{alg:RSN}, we see that $x_{k+1}-x_k=-  t_k P_k^\T M_k^{-1} P_k g_k$.
        From now on, we will show that  $t_k=1$ for $k$ large enough.
	Indeed by \eqref{eq:localt_k}, we have that
	$$f(x_k) - f(x_k+t_k'd_k) + \alpha t_k' g_k^\T d_k
	\ge \frac{\bar{\mathcal{C}}n}{2s} L_H {t_k'}^2\norm{d_k} \left(\frac{c_2s\norm{g_k}^\gamma}{\bar{\mathcal{C}}L_H n\norm{d_k}} - t_k'  \right) \norm{M_k^{-1} P_kg_k}^2. $$
	Hence, by \Cref{ass:local2a} and Lemma \ref{lem:dk}, we deduce that there exists some constant $\mathcal{C}_1$ such that
	$$f(x_k) - f(x_k+t_k'd_k) + \alpha t_k' g_k^\T d_k
	\ge \frac{\bar{\mathcal{C}}n}{2s} L_H {t_k'}^2\norm{d_k} \left(\frac{\mathcal{C}_1}{\|x_k-\bar{x}\|^{1-\gamma}} - t_k'  \right) \norm{M_k^{-1} P_kg_k}^2,$$
	proving that we can take $t'_k=1$ if $\|x_k-\bar{x}\|$ is small enough.\\

	Now notice that for $k$ large enough, $t_k=1$, hence
	\begin{align*}
		\|P_kg_k +P_kH_k(x_{k+1}-x_k)\|&= \|(I_s-P_kH_kP_k^\top(P_kH_kP_k^\top+\eta_kI_s)^{-1})P_kg_k\| \\
		& \le \|\eta_k (P_kH_kP_k^\top+\eta_kI_s)^{-1}P_kg_k \|  \\
		& \le\frac{ \eta_k}{\sigma_{\min}(P_k^\top)}  \|P_k^\top(P_kH_kP_k^\top+\eta_kI_s)^{-1}P_kg_k\| \\
		&=\frac{ \eta_k}{\sigma_{\min}(P_k^\top)} \|d_k\|.
	\end{align*}
Using that $\|\eta_k\|\le c_1\|\Lambda_k\|  + c_2\norm{g_k}^\gamma$ and that $\|g_k\|=O(\|x_k-\bar{x}\|)$ by Lemma \ref{lem:1a}, we deduce, by Lemmas \ref{lem:Lambda_bound} and \ref{lem:dk}, that there exists some constants $\alpha$, $\beta$, $\beta' >0$ such that with probability at least $1-2\exp(-s)$,
	\begin{align*}
		\frac{ \eta_k}{\sigma_{\min}(P_k^\top)} \|d_k\| &\le \frac{1}{\sigma^2_{\min}(P_k^\top)} \left(c_1 \alpha \|x_k-\bar{x}\|^2 + c_2 \beta \|x_k-\bar{x}\|\|g_k\|^{\gamma} \right) \\
		& \le  \frac{1}{\sigma^2_{\min}(P_k^\top)} \left(c_1 \alpha \|x_k-\bar{x}\|^2 + c_2 \beta' \|x_k-\bar{x}\|^{1+\gamma} \right),
	\end{align*}
	where we have used in the second inequality that $\|g_k\| \le O(\|x_k-\bar{x}\|)$.
	Now by \eqref{eq:super1}, \eqref{eq:super2} and the above, we obtain the desired result.
\endproof

Notice that by using \cite{rudelson}, we can furthermore bound $\frac{1}{\sigma_{\min}(P_k^1)}$, with high probability, by $O(\frac{1}{\sqrt{s}-\sqrt{r-1}})$.

Let us consider a function with low dimensionality, i.e. satisfying \eqref{eq:lowdim}. Let us write $\Pi=R^\top R$, where $R \in \mathbb{R}^{s \times n }$ and let us define  $g: y\in \mathbb{R}^s\mapsto f(R^\top)$. Hence, we have that $g(Rx)=f(\Pi x)=f(x)$. By denoting $y_k:=Rx_k \in \mathbb{R}^s$ and assuming that the function $g(y)$ is strongly convex, locally near $\bar{y}:=R\bar{x}$, it is easy to see that \Cref{ass:local2a} is satisfied for the sequence $\{y_k\}$, locally, i.e., there exists $\rho>0$ such that for $k$ large enough;
	$$\|\nabla g(y_k)\| \ge \rho \|y_k-\bar{y}\|$$
	holds. 
Hence, we can prove that there exists some constant $\mathcal{K}>$ such that the following inequality holds with high probability.
	$$ \|y_{k+1}-\bar{y}\| \le \mathcal{K} \|y_k -\bar{y}\|^{1+\gamma} .$$
	By strong convexity of $g(y)$, we know that there exists two constant $l_1>l_2>0$ such that
	$$l_2 (g(y_{k+1}-g(\bar{y})))\le \|y_{k+1}-\bar{y}\|\le l_1 (g(y_{k+1}-g(\bar{y}))).$$
	Hence by following the same proof as in Corollary \ref{cor:2}, we can obtain the following super-linear rate in expectation:
	\begin{theorem}
		 Assume that there exists a function $g: y\in \mathbb{R}^s\mapsto g(y)$ such that $g(Rx)=f(x)$, for some matrix $R \in \mathbb{R}^{s \times n }$ ($s<n$). If the function $g(y)$ is strongly convex, locally near $R\bar{x}$, then there exists a constant  $\mathcal{K}'>0$, such that if $k$ is large enough:
		 \begin{equation}\label{eq:superlinearE}
		 	\mathbb{E}\left[f(x_{k+1})-f(\bar{x})\right] \le \mathcal{K}'	\mathbb{E}\left[f(x_{k})-f(\bar{x})\right]^{1+\gamma},
		 \end{equation}
	\end{theorem}

\section{Numerical illustration}\label{sec:num}
{In this section, we illustrate numerically the randomized subspace regularized Newton method (RS-RNM). 
All results are obtained using Python scripts on a 12th Gen Intel(R) Core(TM) i9-12900HK 2.50 GHz with 64GB of RAM.
As a benchmark, we compare it against the gradient descent method (GD) and the regularized Newton method (RNM) \cite{ueda2010convergence}.
Here we do not aim to prove that our method is faster to the state-of-the-art methods but rather to illustrate the theoretical  results that have been proved in the previous sections.}

\subsection{Support vector regression}
The methods are tested on a support vector regression problem formulated as minimizing sum of a loss function and a regularizer:
\begin{equation}\label{eq:robust_regression}
  f(w) = \frac{1}{m} \sum_{i=1}^{m} \ell(y_i - x_i^\T w) + \lambda \norm{w}^2.
\end{equation}
Here, $(x_i, y_i)\in\R^{n}\times\{0,1\}\ (i=1,2,\dots,m)$ denote the training example and $\ell$ is the loss function. 
$\lambda$ is a constant of the regularizer and is fixed to 0.01 in the numerical experiments below.
We note that \ref{eq:robust_regression} is a type of (generalized) linear model used in the numerical experiments of \cite{gower2019rsn} and \cite{cubicNewton}.
As the loss function $\ell$, we use the following two functions known as robust loss functions: 
the Geman-McClure loss function ($\ell_1$) and the Cauchy loss function  ($\ell_2$) \cite{barron2019general} defined as
\begin{align*}
  \ell_1(t) &= \frac{2t^2}{t^2+4},\\
  \ell_2(t) &= \log\left(\frac{1}{2}t^2+1\right).
\end{align*}
Since both loss functions $\ell_1$ and $\ell_2$ are non-convex,
the objective function \ref{eq:robust_regression} is non-convex.


The search directions at each iteration in GD and RNM are given by
\begin{align*}
  d_k^{\rm GD} &= -\nabla f(w_k),\\
  d_k^{\rm RNM} &= -(\nabla^2 f(w_k) + c_1'\Lambda_k'I_n + c_2'\norm{\nabla f(w_k)}^{\gamma'}I_n)\nabla f(w_k),\\
    &\hspace{10em}(\Lambda'_k = \max(0, -\lambda_{\min}(\nabla^2 f(w_k)))
\end{align*}
and the step sizes are all determined by Armijo backtracking line search \eqref{eq:Armijo_rule} with the same parameters $\alpha$ and $\beta$ for the sake of fairness. 
The parameters shown above and in Section~\ref{sec:RSRNM_algorithm} are fixed as follows:
\begin{equation*}
  c_1=c_1' = 2, c_2=c_2'=1, \gamma = \gamma'=0.5, \alpha = 0.3, \beta= 0.5, s\in \{100,200,400\}.
\end{equation*}

We test the methods on internet advertisements dataset from UCI repository \cite{Dua:2019} that is processed so that the number of instances is 600($=m$) and the number of data attributes is 1500($=n$), 
and the results, until the stop condition $\norm{\nabla f(w_k)}<10^{-4}$ is satisfied, are shown in 
\Cref{fig:RSRNM_result1,fig:RSRNM_result2,fig:RSRNM_result3,fig:RSRNM_result4}.
Our first observation is that RS-RNM converges faster than GD.
GD does not require the calculation of Hessian or its inverse, making the time per iteration small. However, it usually needs a large number of iterations, resulting in slow convergence. 
Next, we look at the comparison between RNM and RS-RNM. 
From \Cref{fig:RSRNM_result1,fig:RSRNM_result2}, we see that RNM has the same or a larger decrease in the function value in one iteration than RS-RNM, and it takes fewer iterations to converge.
This is possibly due to the fact that RNM determines the search direction in full-dimensional space. 
In particular, it should be mentioned that RNM converges rapidly from a certain point on, as it is shown that RNM has a super-linear rate of convergence near a local optimal solution.
However, as shown in \Cref{fig:RSRNM_result3,fig:RSRNM_result4},
since RNM takes a long time to get close to the local solution due to the heavy calculation of the full regularized Hessian, RS-RNM results in faster convergence than RNM. {We also confirm on \Cref{fig:RSRNM_result2} that for small dimensions $s=100,200$  a linear convergence rate seems to be achieved. However for $s=400$ it seems that the method converges super-linearly. }

\begin{figure}[tb]
    \centering
   \begin{minipage}[t]{0.45\linewidth} 
        \centering
	 \includegraphics[width=0.9\columnwidth]{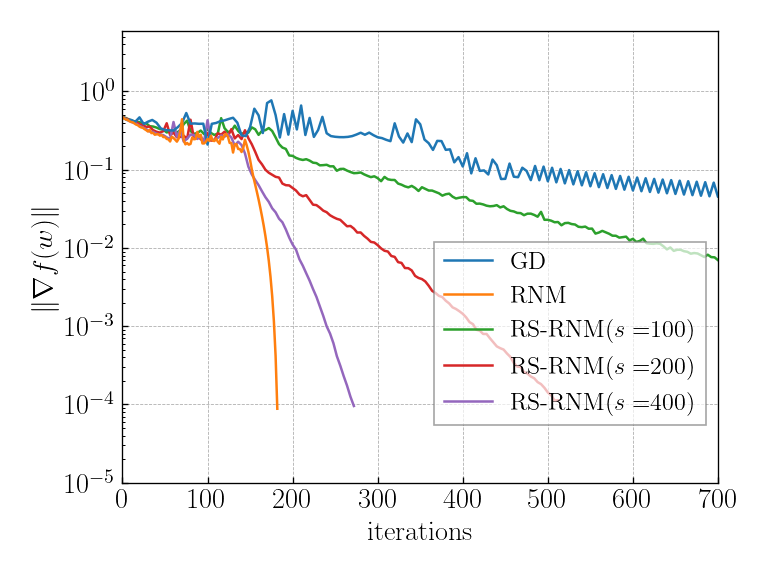}
	 \caption{ Iterations versus $\norm{\nabla f(w)}$ ($\log_{10}$-scale) for Geman-McClure loss}
	 \label{fig:RSRNM_result1}
  \end{minipage}
      \hfill
   \begin{minipage}[t]{0.45\linewidth} 
     \centering
	\includegraphics[width=0.9\columnwidth]{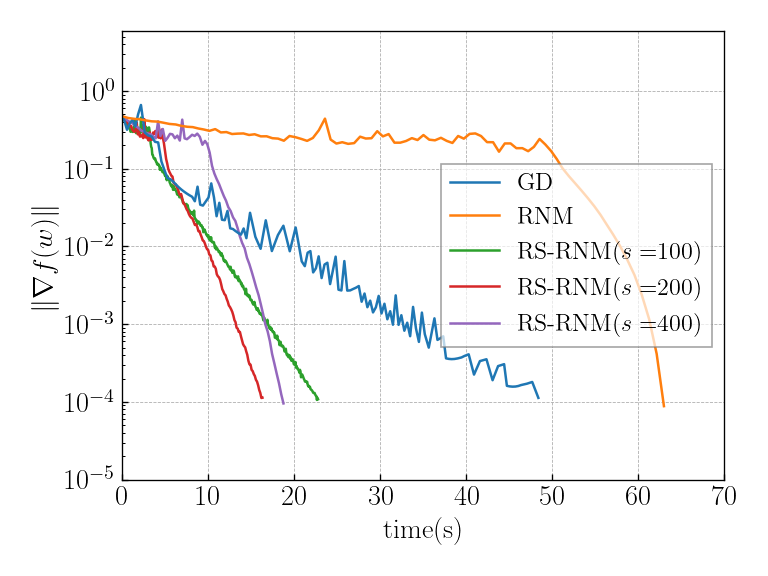}
	\caption{Computation time versus $\norm{\nabla f(w)}$ ($\log_{10}$-scale) for Geman-McClure loss}
	\label{fig:RSRNM_result3}
   \end{minipage}     
  \end{figure}
  
\begin{figure}[tbp]
      \centering
   \begin{minipage}[t]{0.45\linewidth} 
        \centering
	\includegraphics[width=0.9\columnwidth]{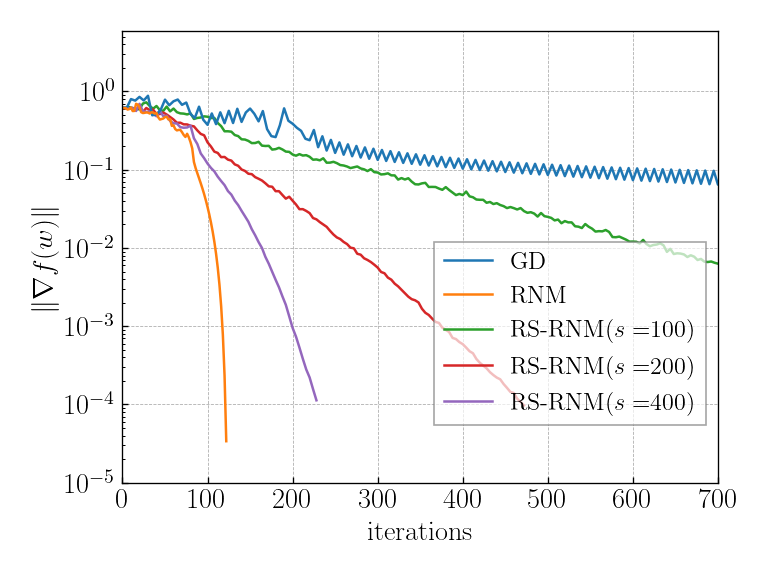}
	\caption{Iterations versus $\norm{\nabla f(w)}$ ($\log_{10}$-scale) for Cauchy loss}
	\label{fig:RSRNM_result2}
    \end{minipage}
      \hfill
   \begin{minipage}[t]{0.45\linewidth} 
     \centering
      \includegraphics[width=0.9\columnwidth]{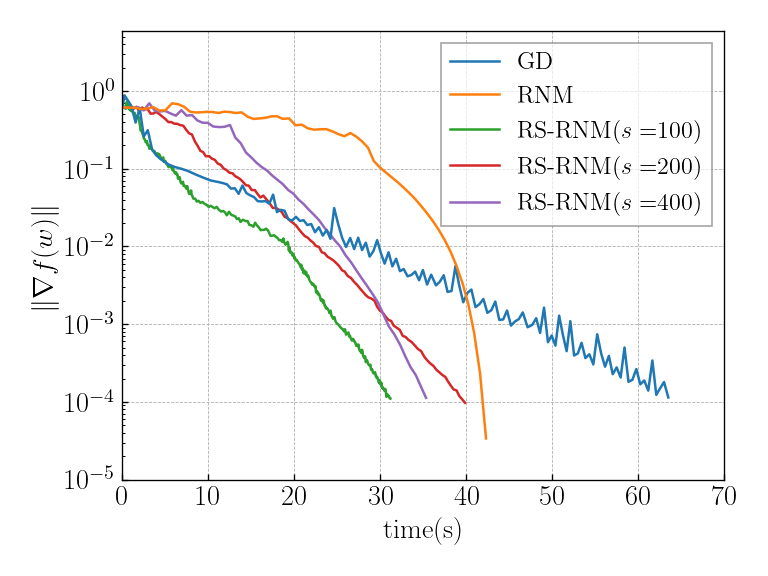}
	\caption{Computation time versus $\norm{\nabla f(w)}$ ($\log_{10}$-scale) for Cauchy loss}
	\label{fig:RSRNM_result4}
    \end{minipage}     
\end{figure}

\subsection{Low rank Rosenbrock function}
To properly illustrate the superlinear convergence proved in the low rank setting (c.f. Section \ref{sec:lowrank}), 
we conducted numerical experiments on a low rank Rosenbrock function:  
$f(x)=R(U^{\top}Ux)$, where 
$$ R(x) = \sum_{i=1}^{n-1}100(x_{i+1}- x_{i}^{2})^{2}+ (x_{i}- 1)^{2},$$
and $U \in \mathbb{R}^{r \times n}$ is a matrix whose columns are orthogonal. If we denote by 
$\Pi \in \mathbb{R}^{n \times n}$ the matrix $U^{\top}U$, we see that for all $x \in \mathbb{R}^n$, $f(x)=f(\Pi x)$, hence the Hessian of $f$ is of rank $r$ for all $x \in \mathbb{R}^n$. The parameters in Section~\ref{sec:RSRNM_algorithm} are fixed as follows:
\begin{equation*}
	c_1=c_1' = 2, c_2=c_2'=1, \gamma = \gamma'=0.5, \alpha = 0.3, \beta= 0.5, s\in \{100,200,600\}.
\end{equation*}

\begin{figure}[tbp]
  \centering
    \begin{minipage}[t]{0.45\linewidth} 
     \centering
	\includegraphics[width=0.9\columnwidth]{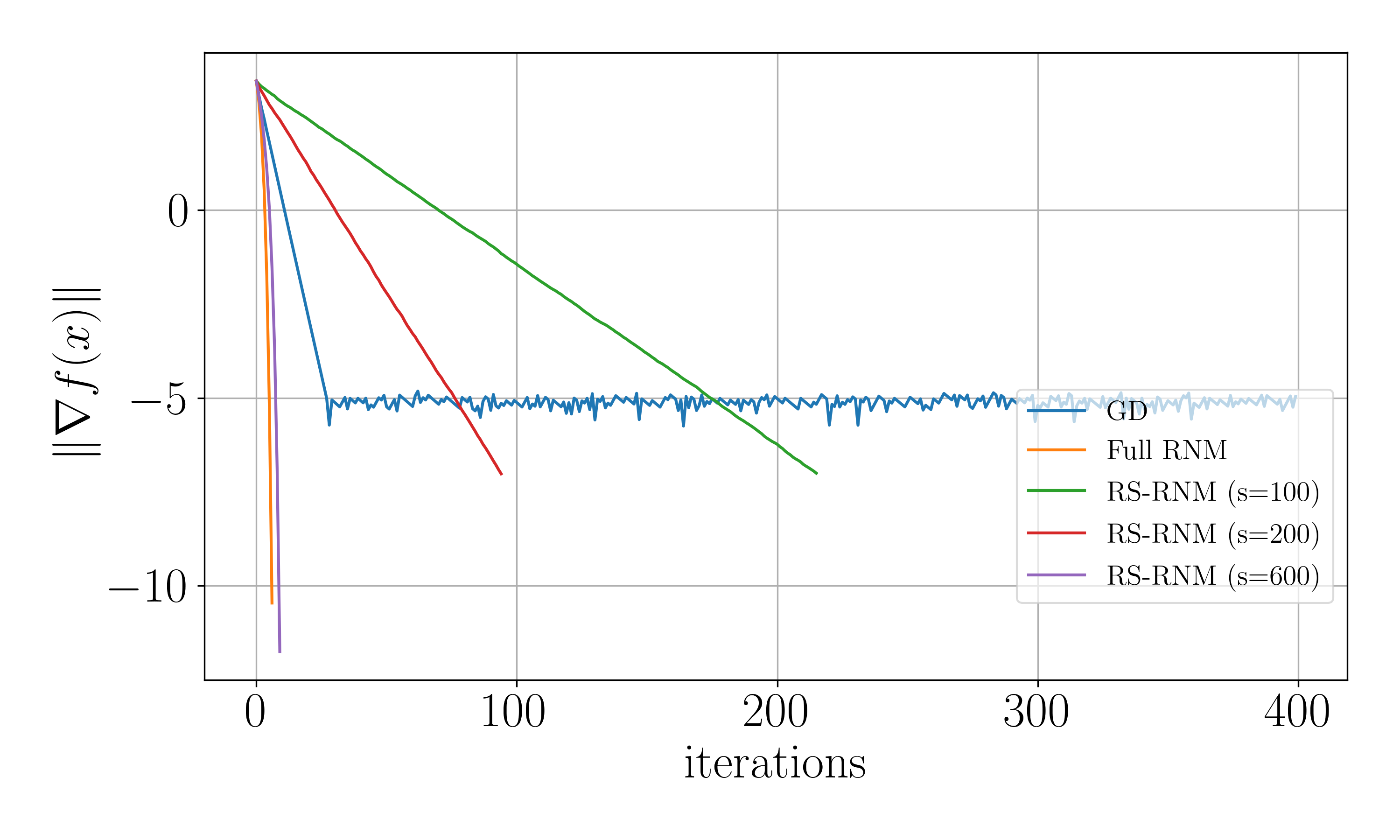}
	\caption{Iterations versus $\norm{\nabla f(x)}$ ($\log_{10}$-scale) for low rank Rosenbrock function}
	\label{fig:Rosen_iter}
    \end{minipage}
       \hfill
   \begin{minipage}[t]{0.45\linewidth} 
        \centering
	\includegraphics[width=0.9\columnwidth]{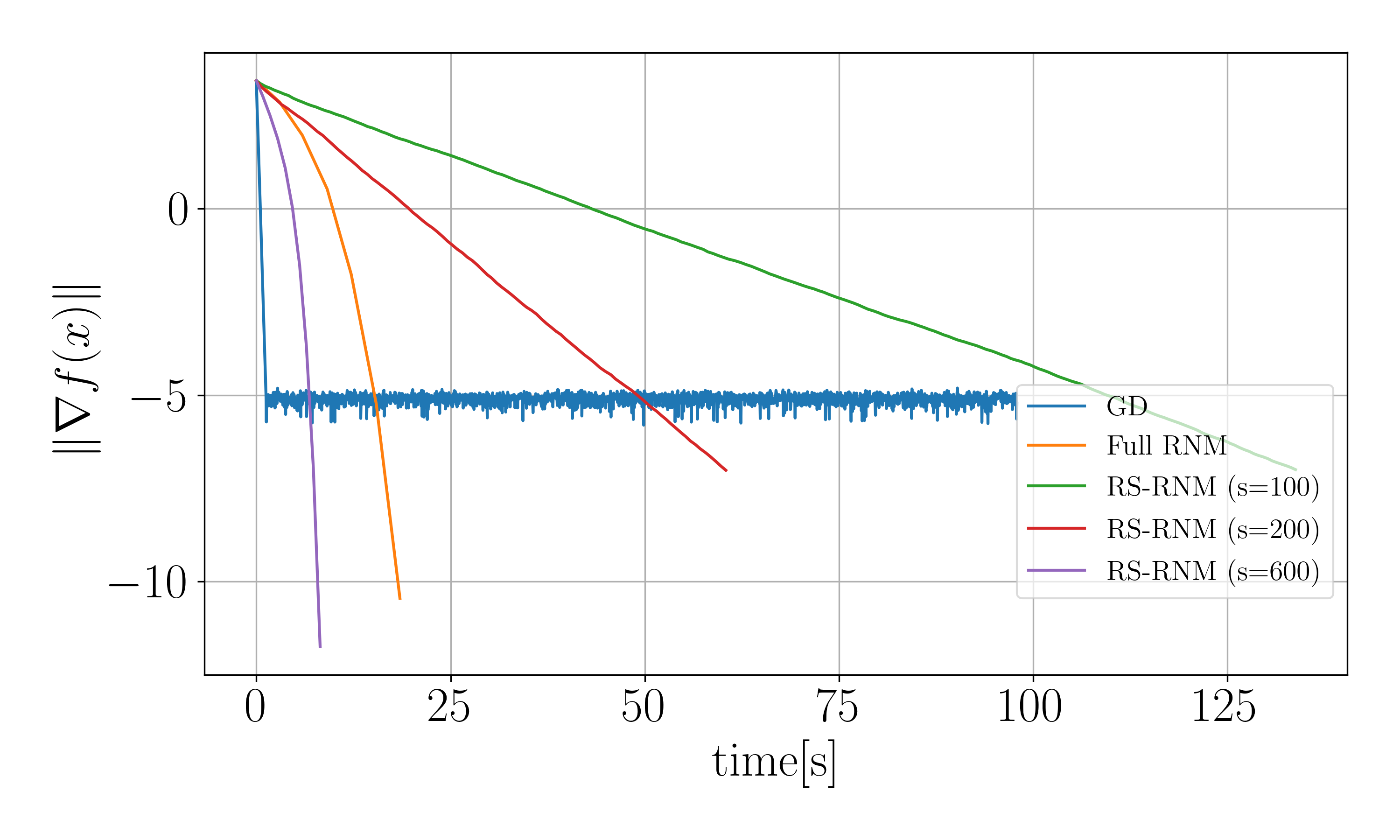}
	\caption{Computation time versus $\norm{\nabla f(x)}$ ($\log_{10}$-scale) for low rank Rosenbrock function}
	\label{fig:Rosen_time}
    \end{minipage}
\end{figure}

\Cref{fig:Rosen_time,fig:Rosen_iter} show experiments for $n=3000$ and $r=500$. 
We selected three values for $s$, two ($s=100,200$) smaller than $r$ and one ($s=600$), larger than $r$. The results confirm the results of Section \ref{sec:local1}: when $s>r$ we have local superlinear convergence, otherwise the convergence is only linear locally.

\subsection{Convolutional neural network}

We tested our method on a micro Convolutional Neural Network (CNN) using the MNIST dataset in \cite{mnist}. We used the cross-entropy loss function 
$m=256$ images. Our CNN is made of the following factors:
\begin{itemize}
	\item one convolutional layer (1 input channel, 1 output channel, kernel size 3),
	\item a ReLU activation,
	\item a max pooling layer (kernel size 2),
	\item a fully connected layer mapping the flattened feature vector to 10 classes.
	\end{itemize}
This setup is intended to demonstrate the differences between the three methods in a controlled, small-scale scenario.
This problem is formulated as 
\begin{equation}
	\min_{w \in \mathbb{R}^n}\frac{1}{m}\sum_{i=1}^m\mathcal{L}(\mathcal{M}(w,x_i),y_i),
	\notag
\end{equation}
where $(x_i,y_i)$ denotes the MNIST dataset with $x_i\in \mathbb{R}^{784}$ and $y_i \in \{0,1\}^{10}$ ($m = 256$), $\mathcal{L}$ denotes the Cross Entropy Loss function, and $\mathcal{M}$ denotes the CNN with $n=1710$ parameters. The parameters  in Section~\ref{sec:RSRNM_algorithm} are fixed as follows:
\begin{equation*}
	c_1=c_1' = 2, c_2=c_2'=1, \gamma = \gamma'=0.5, \alpha = 0.3, \beta= 0.5, s\in \{100,200,500\}.
\end{equation*}

The results are show in \Cref{fig:CNN_time,fig:CNN_iter}.
We notice that our method outperforms GD which is stuck at some stationary point and RNM which is to slow to converge.

\begin{figure}[tbp!]
    \centering
   \begin{minipage}[t]{0.45\linewidth} 
     \centering
	\includegraphics[width=0.9\columnwidth]{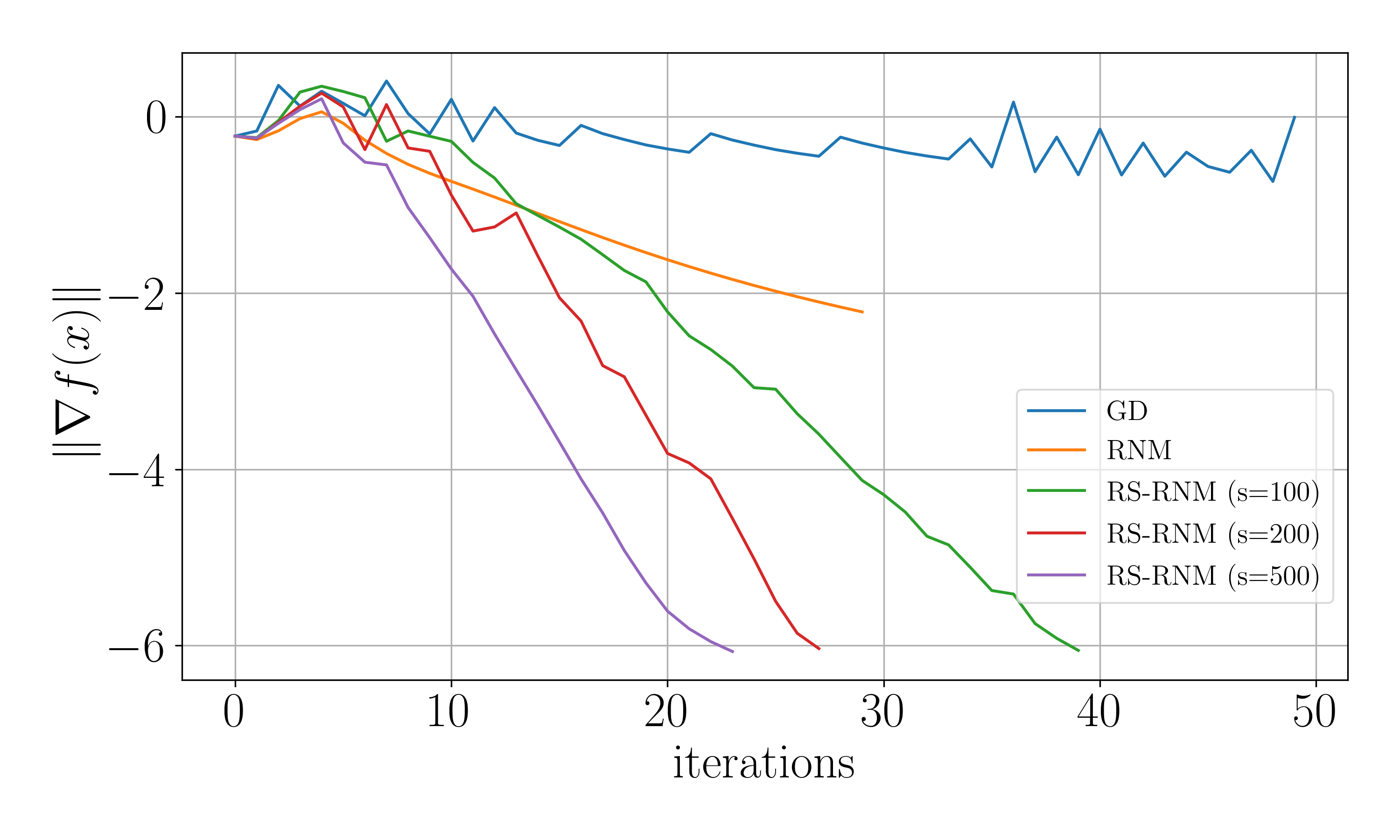}
	\caption{Iterations versus $\norm{\nabla f(w)}$ ($\log_{10}$-scale) for CNN with the MNIST dataset}
	\label{fig:CNN_iter}
   \end{minipage}
       \hfill
   \begin{minipage}[t]{0.45\linewidth} 
        \centering
	\includegraphics[width=0.9\columnwidth]{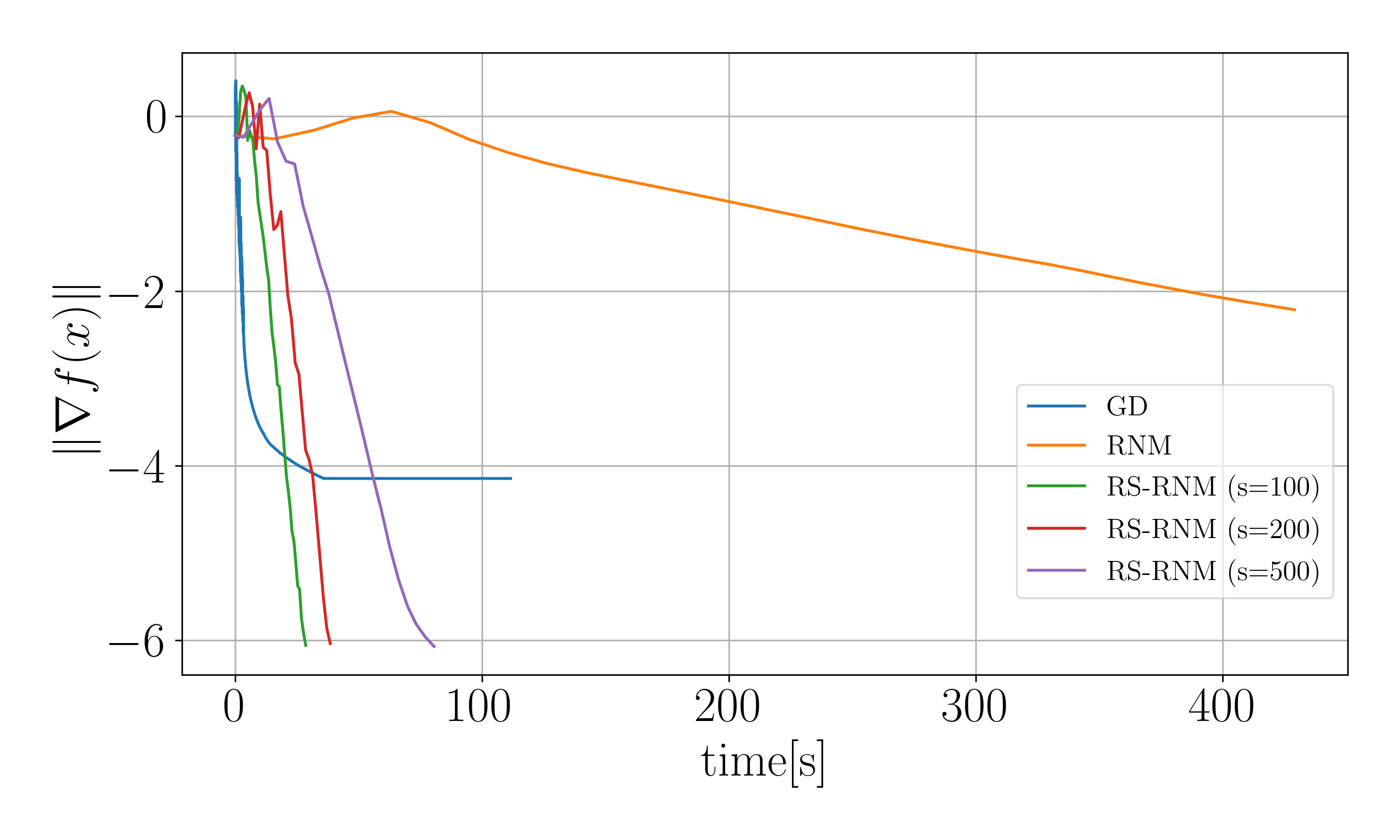}
	\caption{ Computation time versus $\norm{\nabla f(w)}$ ($\log_{10}$-scale) for CNN with the MNIST dataset}
	\label{fig:CNN_time}
         \end{minipage}
\end{figure}

\subsection{Choice of $s$}

In the special case where the Hessian truly has low-rank structure, setting $s$ to this value can substantially speed up convergence, provided the rank is not prohibitively large. However, in more general problems, especially where the Hessian does not exhibit pronounced low-rank properties or its effective rank is unknown, preselecting $s$ is more challenging. One might try to start with some constant value of $s$ and increasing it gradually since the best $s$ ultimately depends on problem-specific characteristics and computational resources.
\color{black}


\section{Conclusions}\label{sec:conclusions}
Random projections have been applied to solve optimization problems in suitable lower-dimensional spaces  
in various existing works.
In this paper, we proposed the randomized subspace regularized Newton method (RS-RNM) for a non-convex twice differentiable function
in the expectation that a framework for the full-space version \cite{ueda2010convergence, ueda2014regularized}  
could be used; 
indeed, we could prove the stochastic variant of the same order of iteration complexity, i.e., the global complexity bound of the algorithm:
the worst-case iteration number $m$ that achieves $\min_{k=0,\ldots,m-1} \norm{\nabla f(x_k)}\le \eps$ 
is $O(\eps^{-2})$
when the objective function has Lipschitz Hessian. 
On the other hand, although RS-RNM uses second-order information similar to the regularized Newton method having 
 a super-linear convergence, we proved that it is not possible, in general, to achieve local super-linear convergence and that local linear convergence is the best rate we can hope for in general. {We were however able to prove super-linear convergence in the particular case where the Hessian is rank deficient at a local minimizer. In this paper we choose to thoroughly investigate local convergence rate for the Newton-based method. One could possibly, in a future work, extend these results to a state-of-the-art second order iterative method and compare the resulting subspace method with other state-of-the-art algorithms, as \cite{gratton,zhou2025,zhu2024}. 



\bibliographystyle{siamplain}
\bibliography{ref}
\end{document}